\documentclass[11pt,a4paper]{article}

\usepackage{enumerate}
\usepackage{geometry,booktabs, fancyhdr}
\usepackage[english]{babel}

\usepackage{amsmath,latexsym,mathrsfs}
\usepackage{amssymb,amsfonts,amscd,amssymb,amsthm}
\usepackage{graphics,verbatim,url}
\usepackage{epic, curves}
\newcommand{\doi}[1]{\url{http://dx.doi.org/#1}}
\usepackage{index}

\newtheorem{theorem}{Theorem}[section]
\newtheorem{lemma}[theorem]{Lemma}

\newtheorem{remark}[theorem]{Remark}
\newtheorem{definition}[theorem]{Definition}

\newcommand{\R}{{\mathbb R}}

\newcommand{\N}{{\mathbb N}}

\newcommand{\mP}{\mathbb P}
\newcommand{\mS}{\mathbb S}
\newcommand{\mE}{\mathbb E}

\newcommand{\cB}{{\mathcal B}}
\newcommand{\cD}{{\cal D}}
\newcommand{\cE}{\mathcal E}
\newcommand{\cH}{{\cal H}}
\newcommand{\cK}{{\cal K}}
\newcommand{\cM}{\mathcal M}

\newcommand{\cO}{\mathcal O}

\newcommand{\cV}{\mathcal V}
\newcommand{\cW}{\mathcal W}

 \newcommand{\cS}{\mathcal S}

\newcommand{\spheroid}{{\Gamma}}

\newcommand{\inpro}[2]{\left\langle{#1},{#2}\right\rangle}
\newcommand{\inprod}[2]{\left\langle{#1},{#2}\right\rangle}
\newcommand{\inprodd}[2]{\left\langle\left\langle{#1},{#2}\right\rangle\right\rangle}

\newcommand{\norm}[2]{\left\|{#1}\right\|_{#2}}
\newcommand{\bnorm}[2]{\Big\|{#1}\Big\|_{#2}}

\newcommand{\snorm}[2]{\left|{#1}\right|_{#2}}

\newcommand{\jum}[1]{\left[#1\right]}
\newcommand{\brac}[1]{\left(#1\right)}

\newcommand{\bsb}[1]{\boldsymbol{#1}}

\newcommand{\G}{\Gamma}

\newcommand{\Om}{\Omega}
\newcommand{\al}{\alpha}

\newcommand{\eps}{\epsilon}
\newcommand{\ga}{\gamma}
\newcommand{\om}{\omega}

\newcommand{\vecx}{\boldsymbol{x}}
\newcommand{\vecy}{\boldsymbol{y}}

\newcommand{\vecn}{\boldsymbol{n}}

\newcommand{\vecu}{\boldsymbol{u}}
\newcommand{\vecv}{\boldsymbol{v}}

\newcommand{\spann}{{\rm{span}}}

\newcommand{\Varr}{{\rm{Var}}}
\newcommand{\Corr}{{\rm{Cor}}}
\newcommand{\Covv}{{\rm{Cov}}}

\newcommand{\goto}{\rightarrow}

\newcommand{\wth}{\widehat}

\newcommand{\slinN}{\sum_{\ell=0}^\infty\,\sum_{m=-\ell}^{\ell}}

\newcommand{\Ylm}{Y_{\ell,m}}

\newcommand{\vhlm}{\widehat{v}_{\ell,m}}
\newcommand{\whlm}{\widehat{w}_{\ell,m}}

\DeclareMathOperator{\divv}{{div \/}}

\usepackage[usenames]{color}
\usepackage{graphicx}
\newcommand{\ac}[1]{{\color{black}{#1}}}
\newcommand{\dpc}[1]{{\color{blue}{#1}}}

\newcommand{\veczero}{\boldsymbol{0}}

\title{A shape calculus based method for a transmission
problem with random interface\footnote{This work has been
funded by BMBF and the Group of Eight Australia
within the DAAD-Go8 Project ``Numerical methods for elliptic 
transmission problems on uncertain interfaces'', Project ID
56266715 and RG123838.}}

\author{Alexey Chernov\thanks{Hausdorff Center for Mathematics and Institute for Numerical Simulation, University of Bonn,
Endenicher Allee 64, Bonn 53115, Germany; }~\thanks{Current address: Department of Mathematics and Statistics, University of Reading, Whiteknights, PO Box 220,
Reading RG6 6AX, United Kingdom, \texttt{a.chernov@reading.ac.uk}}, 
Duong Pham$^\dagger$, 
Thanh Tran\thanks{School of Mathematics and Statistics, 
                The University of New South Wales, 
                Sydney 2052, Australia, \texttt{thanh.tran@unsw.edu.au}‎ }}

\date{\today}
\makeatletter \@addtoreset{equation}{section}
\@addtoreset{table}{section} \makeatother

\begin{document}

\maketitle

\begin{abstract}
\ac{
The present work is devoted to approximation of the statistical moments of the unknown solution of a class of elliptic transmission problems in $\R^3$ with randomly perturbed interfaces. Within this model, the diffusion coefficient has a jump discontinuity across the random transmission interface which models linear diffusion in two different media separated by an uncertain surface. We apply the shape calculus approach to approximate solution's perturbation by the so-called \emph{shape derivative}, correspondingly statistical moments of the solution's perturbation are approximated by the moments of the shape derivative. We characterize the shape derivative as a solution of a related homogeneous transmission problem with nonzero jump conditions which can be solved with the aid of boundary integral equations.
We develop a rigorous theoretical framework for this method, particularly i) extending the method to the case of unbounded domains and ii) closing the gaps and clarifying and adapting results in the existing literature. The theoretical findings are supported by and illustrated in two particular examples.}
\end{abstract}

\section{Introduction}

Elliptic transmission or interface problems arise in many
fields in science and engineering, such as {tomography,} deformation of
an elastic body with inclusions, stationary groundwater flow in heterogeneous medium, 
fluid-structure interaction, scattering of an elastic body and many others.
Combined with the state-of-the-art hardware, 
advanced numerical schemes are capable of producing a highly
accurate and efficient deterministic numerical simulation,
provided that the problem data are known exactly. However,
in real applications, a complete knowledge of the problem parameters is not realistic for many reasons. First, the simulation parameters are often estimated from measurements which can be \emph{inexact} e.g. due to imperfect measurement devices. Second, the parameters are estimated based on a large but \emph{finite} number of system samples (snapshots); this information can be incomplete or stochastic. Finally, parameters of the system 
originate from a mathematical model which is itself only an
approximation of the actual process. Under such
circumstances, \emph{highly accurate results of a single
deterministic simulation for one particular set of
problem parameters are of limited use}. An important
paradigm, becoming rapidly popular over the last
years, see e.g. 
 \cite{BNobTmp07,BarSwbZol11,ChSwb13fokm,CohDVeSwb10,CohDVeSwb11,ForKor10,Git13adap,GraKuoNuySclSlo11,HrbSndSwb08sm,SwbGit11,SwbTod06} 
and references therein, is to treat the lack of knowledge
via modeling uncertain parameters as random fields. If the 
{forward} solution operator is continuous, 
the solution of the forward problem with random parameters 
becomes a well-defined random {field.} Efficient numerical approximation of the random
(or stochastic) solution and its probabilistic
characteristics, e.g. statistical moments, is a highly
non-trivial task representing numerous new interdisciplinary 
challenges: from regularity analysis and numerical analysis to 
modeling and efficient parallel large scale computing.

In this article we develop a deterministic method for numerical 
solution for a class of transmission problems with randomly 
perturbed interfaces. The equation to be solved is of the
form
\[
- \nabla\cdot(\al\nabla u) = f \quad\text{in } D_{\pm},
\]
where $D_{-}$ is a random bounded domain in 
$\R^3$ and $D_+ = \R^3 \setminus \overline{D_-}$ is its complement. 
The domains share a common random surface $\G$, and the
coefficient function $\al$ takes (in general) distinct
constant values in $D_-$ and $D_+$, respectively. The solution $u$ is subject to jump conditions across
$\G$. A precise description of the model problem is deferred until 
Section~\ref{subsec:pro}, where {a probabilistic perturbation model for the surface $\G$ 
(and thus $D_{\pm}$) will be rigorously introduced.} 
{Within this model, the transmission interface depends on the ``random event'' $\om$ 
and the parameter $\eps\geq 0$ controlling the amplitude 
of the perturbation.} Therefore, the solution $u$ depends on
 $\om$ and $\eps$, and will be denoted by $u^\eps(\om)$.
 {The case $\varepsilon = 0$ corresponds to the \emph{zero perturbation}. 
In the present paper we are aiming at estimating probabilistic properties of the solution perturbation $u^\eps(\om) - u^0$ when the perturbation parameter is small, $\eps\ll 1$.}

More precisely, we exploit the ideas from the recent publications 
\cite{Ch08sens,ChSwb13fokm,Hrb10output,HrbLi13,HrbSndSwb08sm}
and propose to approximate the statistical moments of the
solution perturbation by the moments of the linearized
solution, i.e. for a fixed (small) value of the perturbation
parameter $\epsilon$ the $k$-th order statistical moments 
of the solution perturbation are approximated by

\begin{equation}\label{equ:Mk u}
 \cM^k[u^\epsilon - u^0] \approx \epsilon^k
 \cM^k[u']
\end{equation}
{
and similarly
\begin{equation}
\cM^k[u^\epsilon - \mE[u^\epsilon]]\approx \epsilon^k
 \cM^k[u'].
\end{equation}
}

Here $u'$ is the \emph{shape derivative} of $u^\eps$
\emph{formally} understood as the linear order term in the asymptotic expansion
\begin{equation}\label{ShapeDeriv-formal}
u^\epsilon(\vecx,\omega) = u^0(\vecx) +\epsilon  u'(\vecx,\omega) + \cdots,
\qquad \epsilon \to 0,
\end{equation}
for almost all random events $\omega \in \Omega$ at a certain 
\emph{fixed} point $\vecx$ in the Euclidean space $\R^3$.
The notion of the shape derivative has been introduced 
in the context of the shape optimization (see e.g. the 
monograph \cite{SokZol92} and the references therein) and allows to quantify sensitivity of the solution of a PDE to small perturbation of the boundary. 

Although very intuitive,
\eqref{ShapeDeriv-formal} cannot be used as a rigorous
definition of $u'(\vecx,\omega)$. In particular, convergence of
the asymptotic expansion and herewith the existence of the
shape derivative is unclear. 
In the first part of this article (Section~\ref{sec:Shape
cal}) we develop a rigorous mathematical 
theory of existence of the shape derivative for the class of 
elliptic transmission problems under consideration. 
Similarly to \cite[Lemma 1]{HrbLi13}, we obtain a characterization 
of the shape derivative $u'(\vecx,\omega)$ as a solution of a 
deterministic transmission problem on a fixed interface. 
Our contribution in this section is two-fold: \textit{i)} we extend the 
notion of shape derivatives to the case of unbounded
domains,
and \textit{ii)} we fill the gaps and unclarities in existing 
literature where no rigorous discussion on existence of shape
derivatives is presented.



As mentioned above, for almost all $\omega \in \Omega$ the
shape derivative $u'(\cdot,\omega)$ is a solution of a
deterministic problem in $\mathbb R^3$ with (in general)
nonhomogeneous jump conditions but with vanishing volume source
term. 
The second contribution of this article is the analysis of
boundary element methods \cite{HsiWen08, SauSch11,
Steinbach08} which are used to solve this transmission
problem on deterministic domains with deterministic
interface. A tensorization argument is then used to obtain the approximation~\eqref{equ:Mk u} for the statistical moments.

Finally, we illustrate the accuracy of the linearization approach by 
considering two examples setting on the unit
sphere $\Gamma := \{|\vecx| = 1\}$ with uniform radial perturbation. 
The first example involves a pre-determined solution with 
radial symmetry, so that the exact and the linearized solutions 
as well as their second moments are available explicitly. 
We observe that in this particular case the 
linearization error for the second order statistical moments 
is of the order $\mathcal O(\epsilon^4)$ rather than $o(\epsilon^2)$ as confirmed by the theory. 
The second example involves non-symmetric data so that the 
linearized solution is not available explicitly. 
To solve this problem numerically we use the sparse spectral 
tensor product BEM developed in \cite{ChPham12_prep}. 
This method exploits the underlying geometry of the 
formulation and uses the basis of spherical harmonics 
being the eigenfunctions of the integral operator governing the problem.


The paper is organized as follows. Section \ref{sec:model problem} 
contains the description of the random surface perturbation 
model and the rigorous formulation of the model transmission 
problem, preceded by the details on the function spaces 
involved in the analysis. Section \ref{sec:Shape cal} 
contains the generalization of the shape calculus to the 
case of unbounded domains, definition and characterization 
of the material and shape derivatives for the underlying 
model transmission problem and a rigorous proof and 
error bounds for the approximation \eqref{equ:Mk u}.
Section \ref{sec:boundary reduction} contains the details 
of the boundary reduction for the linearized problem. 
Section \ref{sec:Examp} contains two examples, an 
analytic and a numerical, illustrating the accuracy of the method.

\section{Model elliptic transmission problem on a random 
interface}\label{sec:model problem}

{We start with some preliminary definitions and notations in Section \ref{subsec:statmom}. 
Section \ref{sec:sto int} contains the description of a model 
for the random surface perturbation. We introduce the 
randomized model problem in the strong form in Section \ref{subsec:pro}. 
The details on Sobolev spaces involved are summarized in Section \ref{subsec:Sobolev}. }

\subsection{{Bochner spaces and} statistical moments} \label{subsec:statmom}

Throughout this paper we denote by $(\Om, \Sigma,\mP)$ a generic complete probability space and let $X$ be a separable Hilbert space. For any $1\leq k \leq \infty$, the Bochner space $L^k(\Om, X)$ is defined as usual by
\begin{equation}\label{Lk-def}
L^k(\Omega,X): = 
\big\{v: \Omega\goto X, \text{ measurable}: \|v\|_{L^k(\Omega,X)} <\infty \big\}
\end{equation}
with the norm
\begin{equation}\label{Lk-norm-def}
\norm{v}{L^k(\Omega, X)} :=
\left\{
\begin{array}{cr} \displaystyle
\left(\int_{\Omega}\norm{v(\omega)}{X}^k d\mathbb
P(\omega)\right)^{1/k}, & 1\leq k < \infty,\\[3ex]
\mathop{\rm ess sup}\limits_{\omega\in \Omega}\norm{v(\omega)}{X}, & k=\infty.
\end{array}
\right.
\end{equation}
The elements of $L^k(\Omega,X)$ are called \emph{random fields}. 
We remark that a part of the subsequent analysis can be carried 
out in the more general case when $X$ is a general Banach space, 
cf. \cite{ChSwb13fokm}; in this paper we restrict to the Hilbertian 
setting which is sufficient for the purpose of this work. In particular, 
when $X_1$ and $X_2$ are two separable Hilbert spaces, their tensor 
product $X_1 \otimes X_2$ is a separable Hilbert space with the natural 
inner product extended by linearity from $\langle v\otimes a, 
w\otimes b\rangle_{X_1 \otimes X_2} = \langle v, w\rangle_{X_1}
\langle a, b\rangle_{X_2}$, cf. e.g. \cite[p. 20]{LigChe85}, 
\cite[Definition 12.3.2, p.298]{Aubin00}. In this {paper} we work 
with $k$-fold tensor products
\begin{equation}\label{Xk-def}
 X^{(k)} := X \otimes \dots \otimes X.
\end{equation}
with the natural inner product satisfying $\langle v_1\otimes \dots\otimes v_k, w_1\otimes \dots\otimes w_k\rangle_{X^{(k)}} = \langle v_1, w_1\rangle_{X} \dots \langle v_k, w_k\rangle_{X}$.

\begin{definition}
For a random field $v\in L^k(\Omega, X)$\dpc{,} 
its $k$-order moment $\cM^k[v]$ is an element of $X^{(k)}$ defined by
\begin{equation}\label{moments-def}
\cM^k [v] :=  \int_{\Omega} 
\big(
\underbrace{v(\omega)\otimes\cdots\otimes v(\omega)}_{k\textrm{-times}}
\big)
\,d\mathbb P(\omega).
\end{equation}
\end{definition}

In the case $k=1$, the statistical moment $\cM^1 [v]$ coincides with the \textit{mean value}
of $v$ and is denoted by $\mE [v]$. If $k\geq 2$, 
the statistical moment $\cM^k [v]$ is the 
\textit{$k$-point autocorrelation function} of $v$.
The quantity $\cM^k [v - \mE[v]]$ is termed the $k$-th central moment of $v$. We distinguish in particular 
second order moments: the \textit{correlation} and \textit{covariance} defined by
\begin{equation}
\Corr[v] := \cM^2[v], \qquad \Covv[v] := \cM^2[v - \mE[v]].
\end{equation}
In this paper we work with $X$ being Sobolev spaces of real-valued functions defined on a domain $U\subset\R^3$
yielding, in  particular, the representation
\begin{equation}\label{Corr-def}
{\Corr[v](\vecx,\vecy)}
:=
\int_{\Om}
v(\vecx,\om) 
v(\vecy,\om)
\,d\mP(\om),
\quad
\vecx,\vecy\in U.
\end{equation}
We observe that $\Corr[v]$ is defined on the Cartesian product $U \times U$. 
Similarly, $\cM^k[v]$ is defined on the $k$-fold Cartesian 
product $U \times \dots \times U$. {Here,} 
the dimension of the underlying domain grows rapidly with increasing moment order $k$.

\subsection{Random interfaces}\label{sec:sto int}

{Consider a fixed bounded domain $D^0_{-} \subset \R^3$ and 
let $D^0_+ := \R^3 \setminus \overline{D^0_{-}}$ be its 
complement. Then the interface $\G^0 = \overline{D^0_{-}} 
\cap \overline{D^0_{+}}$ is a closed manifold in $\R^3$. 
For the subsequent analysis we assume that $\G^0$ is at 
least of the class $C^{1,1}$. This implies
that the outward normal vector
$\bsb n^0$ to $\G^0$ is Lipshitz continuous: $\bsb n^0 \in C^{0,1}(\G^0)$. 
The partition $\R^3 = \overline{D^0_{+}} \cup \overline{D^0_{-}}$ 
and the interface $\G^0$ will be fixed throughout the paper and 
will be called the \emph{nominal partition} and \emph{nominal interface}, respectively.

In the present paper we utilize the domain perturbation model 
based {on} the \emph{speed method} (see e.g. 
the monograph \cite{SokZol92} and references therein) and random domain perturbation model from \cite{Ch08sens,ChSwb13fokm,Hrb10output,HrbLi13,HrbSndSwb08sm}. Suppose $\kappa \in L^k(\Omega, \ac{C^{0,1}(\G^0)})$ is a random field, i.e. for almost any realization $\om \in \Om$, we have $\kappa(\cdot,\om) \in \ac{C^{0,1}(\G^0)}$. For some sufficiently small, nonnegative $\eps$ we consider a family of random interfaces of} the form
\begin{equation}\label{equ:rand inter}
\G^\eps(\om)
=
\{
\vecx 
+
\eps
\kappa(\vecx,\om)\vecn^0(\vecx):
\vecx\in \G^0
\},
\quad
\om\in\Om\dpc{.}
\end{equation}
Here, the uncertainty of the surfaces $\G^\eps(\om)$ is represented by the uncertainty in $\kappa(\cdot,\om)$. 
Notice that the interface $\G^\eps(\om)|_{\eps = 0}$ is 
identical with $\G^0$ and therefore is a deterministic 
closed manifold. Moreover, 
the limit $\G^\eps(\omega) \to \G^0$ as $\eps \to 0$ is well defined in $L^k(\Om,C^{0,1})$. 
If we identify $\G^\eps$ and $\G^0$ with their graphs, then
\begin{equation}\label{Gamma-conv}
\begin{split}
\|\G^\eps - \Gamma^0\|_{L^k(\Om,C^{0,1})} 
&
= \eps \left( 
 \int_\Om\|\kappa(\cdot,\omega) \vecn^0\|_{C^{0,1}(\G^0)}^k \, d \mP(\om) \right)^{\frac{1}{k}}
\leq \ac{2}\eps \|\kappa\|_{L^k(\Om,\ac{C^{0,1}(\Gamma^0)})}\|\vecn^0\|_{C^{0,1}(\G^0)}.
\end{split}
\end{equation}
This implies that for almost all $\om \in \Om$ and a sufficiently small $\eps\geq 0$ the surface $\G^\eps(\omega)$ is a Lipshitz continuous closed manifold separating the interior domain $D^\eps_-(\om)$ and its complement $D^\eps_+(\om) := \R^3 \setminus \overline{D^\eps_-}$. \ac{The \emph{shape calculus} in Section \ref{sec:Shape cal} requires a somewhat stronger smoothness assumption on $\kappa$, namely that the realizations of $\kappa$ belong to $C^{1}(\G^0)$.}
From \eqref{equ:rand inter} we observe that the \emph{mean random interface} is represented by
\[
\mE[\G^\eps]
=
\big\{ \vecx + \eps {\mE[\kappa(\vecx,\cdot)] \bsb n^0}(\vecx), \ \vecx\in\G^0 \big\}.
\]
Without loss of generality, we may assume 
that the random perturbation amplitude $\kappa(\vecx,\om)$ is centered, i.e.,
\begin{equation}\label{equ:kappa sym}
{\mE[\kappa(\vecx,\cdot)]}
=
0
\qquad
\forall
\vecx\in\G^0.
\end{equation}
{In this case}
\[
\mE[\G^\eps]
= \G^0 \qquad \text{and} \qquad
\Covv[\kappa](\vecx,\vecy)
=
\Corr[\kappa](\vecx,\vecy).
\]

\subsection{The model problem}\label{subsec:pro}

As shown above, for a sufficiently small value 
$\eps\geq 0$ the surface perturbation model \eqref{equ:rand inter} 
generates a well defined partition of $\R^3$ into a 
bounded Lipshitz domain $D^\eps_-(\om)$ and its complement 
$D^\eps_+(\om) = \R^3 \setminus \overline{D^\eps_-}$ separated 
by the closed Lipshitz manifold $\G^\eps(\om) = \overline{D^\eps_-(\om)} \cap \overline{D^\eps_+(\om)}$. We consider a piecewise constant diffusion function subjected to this partition:
\begin{equation}\label{equ:alpha cond}
\alpha^\eps(\vecx,\omega)
=
\begin{cases}
\alpha_-, & \vecx\in D^\eps_{-}(\omega),
\\
\alpha_+, & \vecx\in D^\eps_{+}(\omega),
\end{cases}
\end{equation}
where $\alpha_-$ and $\alpha_+$ are two positive constants
independent of $\vecx$, $\eps$, and $\omega$. 
Having this we introduce the model elliptic transmission 
problem as a problem of finding $u^\eps$ satisfying
\begin{subequations}\label{equ:original prob}
\begin{align}
-\nabla\cdot\big(\alpha^\eps(\vecx,\omega)\nabla u^\eps(\vecx,\omega)\big) &
= f(\vecx)
\quad\text{in}\ D^\eps_\pm(\omega),
\label{equ:lap equ}
\\
[u^\eps(\vecx,\omega)]
&
=
0
\quad
\text{on}\ \Gamma^\eps(\omega),
\label{equ:tranmiss 1}
\\
\left[
\alpha^\eps(\vecx,\omega)
\frac{\partial u^\eps}{\partial\vecn}(\vecx,\omega)
\right]
&
= 0
\quad
\text{on}\
\Gamma^\eps(\omega),
\label{equ:tranmiss 2}
\\
u^\eps(\vecx,\omega) &
= O(\snorm{\vecx}{×}^{-1})
\quad\text{as}\ \snorm{\vecx}{×}\goto+\infty.
\label{equ:inft cond 1}
\end{align}
\end{subequations}

Here, $\partial/\partial\vecn$ denotes the normal derivative 
{on $\Gamma^\eps(\om)$, 
i.e. $\partial/\partial\vecn = \vecn^\eps(\vecx,\omega) \cdot \nabla$},
{where $\vecn^\eps(\vecx,\omega)$ is the unit normal vector to the interface
$\Gamma^\eps(\omega)$ pointing into the interior of
$D^\eps_{+}(\omega)$.}
Let $u^\eps_{-}(\omega)$ and
$u^\eps_{+}(\omega)$ be the restrictions of 
$u^\eps(\omega)$ on $D^\eps_{-}(\omega)$ and 
$D^\eps_{+}(\omega)$, respectively. Then the jump $[u^\eps(\omega)]$ is understood to be 
$u^\eps_-(\omega)-u^\eps_+(\omega)$
on $\Gamma^\eps(\omega)$ in the sense of trace for each sample $\omega$. 
Similarly 
\[
\left[
\alpha^\eps(\vecx,\omega)
\frac{\partial u^\eps}{\partial\vecn}(\vecx,\omega)
\right]
=
\alpha^\eps_-
\frac{\partial u^\eps_-}{\partial\vecn}(\vecx,\omega)
-
\alpha^\eps_+
\frac{\partial u^\eps_+}{\partial\vecn}(\vecx,\omega),
\quad
\vecx\in \G^\eps(\om).
\]
The function $f\in H^1(\R^3)$ is assumed to be independent 
of $\omega$ and {thereby represents a deterministic source} function in $\R^3$. 

{The model problem \eqref{equ:lap equ}--\eqref{equ:inft cond 1} 
represents a stationary diffusion in $\R^3$ with piecewise constant 
diffusivity in the interior and exterior domain. The uncertainty in 
the random solution $u^\eps(\vecx,\om)$ is implied by the uncertain 
location of the transmission interface $\G^\eps(\omega)$. The solution depends nonlinearly 
on the interface and a linearization process will first be used to 
linearize the initial problem. The tool in this process is shape 
calculus which will be presented in Section~\ref{sec:Shape cal}.
In what follows we address the problem of approximation of the 
statistical moments
\begin{equation}\label{moments}
\mE[u^\eps],\quad \cM^k[u^\eps - u^0], \quad \text{and} \quad 
\cM^k[u^\eps - \mE[u^\eps]], \quad k \geq 2,
\end{equation}
{with} this strategy and the rigorous control of the approximation error.

\subsection{Sobolev spaces}\label{subsec:Sobolev}

In this section we introduce function spaces 
needed for the forthcoming analysis. These spaces 
will allow to identify the unique weak {solution} of the 
model problem \eqref{equ:lap equ}--\eqref{equ:inft cond 1} 
and characterize the moments \eqref{moments}.

Let $\mathcal{G}$ be a sphere-like surface, i.e.,
there exists a diffeomorphism 
 $\rho:\mS\goto\mathcal{G}$ such that 
\[
\mathcal{G}
=
\{
\rho(\vecx) : \vecx\in \mS
\}.
\]
Here, $\mS$ is the unit sphere in $\R^3$.
The surface $\mathcal{G}$ divides $\R^3$ into two subdomains, a bounded domain 
$D_-$ and an unbounded domain $D_+$.
For any distribution $v$ defined on $\mathcal{G}$, and for any point $\rho(\vecx)$ 
on $\mathcal{G}$, we can write 
\[
{(v \circ \rho) (\vecx) = }
v(\rho(\vecx))
=
\slinN
\vhlm
\Ylm(\vecx),
\]
where 
\begin{equation}\label{equ:vhlm def G}
\vhlm
=
\int_{\mS}
(v\circ\rho)(\vecx)\,\Ylm(\vecx)\,d\sigma_{\vecx}
\end{equation}
are the Fourier coefficients of $v$. 
Here $\Ylm$ are spherical harmonics, which are the restrictions on the unit
sphere  $\mS$ of  homogeneous harmonics polynomials in $\R^3$.
The Sobolev space $H^s(\mathcal{G})$, for 
$s\in\R$, is defined by
\begin{equation}\label{equ:Sobolev def}
H^s(\mathcal{G})
=
\bigg\{ 
v \in \cD'(\mathcal{G}) : 
\slinN
(1+\ell)^{2s}
\snorm{\vhlm}{×}^2
<
+\infty
\bigg\},
\end{equation}
where $\cD'(\mathcal{G})$ is the set of distributions on $\mathcal{G}$.
The corresponding inner product and the norm are given by
\begin{equation}\label{equ:inner prod}
\inpro{v}{w}_{H^s(\mathcal{G})}
=
\slinN(1+\ell)^{2s}
\vhlm \whlm,
\quad
v,w\in H^s(\mathcal{G}),
\end{equation}
and 
\begin{equation}\label{equ:norm Sobolev}
\norm{v}{H^{s}(\mathcal{G})}
=
\brac{\slinN (1+\ell)^{2s}\snorm{\vhlm}{×}^2}^{1/2},
\quad
v\in H^s(\mathcal{G}).
\end{equation}
We note here that the inner product~\eqref{equ:inner prod} and 
the norm~\eqref{equ:norm Sobolev} satisfy
\begin{equation}\label{equ:G mS rel}
\inpro{v}{w}_{H^s(\mathcal{G})}
=
\inpro{v\circ\rho}{w\circ\rho}_{H^s(\mS)}
\quad
\text{and}
\quad
\norm{v}{H^s(\mathcal{G})}
=
\norm{v\circ\rho}{H^s(\mS)}
\end{equation}
for any $v,w\in H^s(\mathcal{G})$. 
The set 
$\{\Ylm\circ\rho^{-1}: \ell\in\N,{\ m} = -\ell,\ldots,\ell\}$
is an orthogonal basis for $H^s(\mathcal{G})$.
We also note that the space $H^0(\mathcal{G})$ {can be understood as} a 
weighted $L_2$-space on the interface $\mathcal{G}$.

We now introduce the tensor product of Sobolev spaces 
{on the $k$-fold Cartesian product domains 
$\mathcal G^k = \mathcal G \times \dots \times \mathcal G$. 
These spaces will be used later on for characterization of 
statistical moments. By boldface symbols we denote 
multiindices with $k$ integer components, e.g. $\bsb \ell = (\ell_1,\dots,\ell_k)$.
Given $s\in\R$, the Sobolev space $H_{\rm{mix}}^{s}(\mathcal{G}^k)$ is defined to be the space of all distributions $v(\vecy_1,\dots,\vecy_k)$ with $\vecy_1,\dots,\vecy_k \in \mathcal{G}
$
satisfying
\begin{equation}\label{equ:inn ten def}
\begin{split}
\norm{v}{H_{\rm{mix}}^{s}(\mathcal{G}^k)} 
&:= \inpro{v}{v}_{H_{\rm{mix}}^{s}(\mathcal{G}^k)}^{1/2} < \infty,\\
\inpro{v}{w}_{H_{\rm{mix}}^{s}(\mathcal{G}^k)}
&:=
\sum_{\bsb \ell = 0}^\infty
\sum_{
\bsb m = -\bsb \ell
}^{\bsb \ell}
\left(
\prod_{i=1}^k
(1+\ell_i)^{2s}
\right)
\wth v_{\bsb \ell,\bsb m}
\wth w_{\bsb \ell,\bsb m}
\end{split}
\end{equation}
}
with the Fourier coefficients 
\begin{equation}\label{equ:ab 7}
\wth v_{\bsb \ell,\bsb m}
:=
\int_{\vecx_1 \in \mS} \dots 
\int_{\vecx_k \in \mS}
v(\rho(\vecx_1),\dots,\rho(\vecx_k))
\, \left(\prod_{i=1}^k Y_{\ell_i,m_i}(\vecx_i) \right)
\,d\sigma_{\vecx_1}
\dots
d\sigma_{\vecx_k}
\end{equation}
{Recalling definition \eqref{Xk-def} we observe 
that $H^s_{\rm mix} (\mathcal{G}^k)$ is isometrically 
isomorphic to the tensor product space $H^s(\mathcal{G})^{(k)}$. 
These spaces will be identified in what follows. \ac{We also use the notation $H^s_{\rm mix}(K^k)$ for the tensor product $H^s(K)^{(k)}$ where $K$ is a compact subset of $\R^3$.}

Sobolev spaces on bounded domains in $\R^3$ are defined, as usual, as
spaces of all distributions whose partial 
derivatives are square integrable.
Proper treatment of the transmission problem \eqref{equ:lap equ}--\eqref{equ:inft cond 1} in unbounded domains in $\R^3$ requires a special care. Following \cite{SauSch11}, for an
unbounded domain $U \subset \R^3$ we introduce the space} 
\begin{equation}\label{equ:W Sobolev}
H_w^1(U)
:=
\bigg\{
v\in\cD'(U) : 
\norm{v}{H_w^1(U)}
=
\brac{\int_{U}\Big{(} \snorm{\nabla  v}{×}^2
+
\frac{\snorm{v(\vecx)}{×}^2}{1+\snorm{\vecx}{×}^2}\Big{)}\,d\vecx}^{1/2}
<
+\infty
\bigg\}.
\end{equation}
{Specifically, for a given partition $\R^3 = \overline{D^\eps_-} 
\cup \overline{D^\eps_+}$ we define the space
\begin{equation}\label{equ:W def}
W_\eps
:=
\big\{
v = (v_-,v_+)\in H^1(D^\eps_-)\times H_w^1(D^\eps_+) : \jum{v}_{\Gamma^\eps}
=
0
\big\}
\end{equation}
which is a weighted Sobolev space on $D^\eps_-\cup D^\eps_+$
with corresponding norm and seminorm
\begin{equation}\label{equ:W norm}
\norm{v}{W_\eps}
:=
\left(
\norm{v_-}{H^1(D^\eps_-)}^2
+
\norm{v_+}{H_w^1(D^\eps_+)}^2
\right)^{1/2},
\quad
\snorm{v}{W_\eps}
:=
\left(
\int_{D^\eps_-}
\snorm{\nabla v_-}{}^2\,
d\vecx
+
\int_{D^\eps_+}
\snorm{\nabla v_+}{}^2\,
d\vecx
\right)^{1/2}.
\end{equation}
}

The following lemma which will be frequently used in 
the rest of the paper states the equivalence between the 
norm $\norm{\cdot}{W_\eps}$ and seminorm {$\snorm{\cdot}{W_\eps}$. 
The proof of this result follows by the Friedrichs} inequality and 
the technique in the proof of~\cite[Theorem 2.10.10]{SauSch11}.

\begin{lemma}\label{lem:nor snor}
The seminorm {$\snorm{\cdot}{W_\eps}$} is also a norm in {$W_\eps$} which is equivalent to
{$\norm{\cdot}{W_\eps}$}.
\end{lemma}


\section{Shape calculus}\label{sec:Shape cal}

{The aim of the present section is the systematic development of the
linearization theory for the solution $u^\eps$ of the model problem
\eqref{equ:lap equ}--\eqref{equ:inft cond 1} with respect to the shape of
the perturbed interface $\G^\eps$. This techniques is also known as
\emph{shape calculus} and originates from shape optimization; see \cite{SokZol92} and references therein. For this purpose, in} the first three subsections that follow, we temporarily stay 
away from randomness and consider only deterministic perturbed interfaces.
\subsection{Perturbation of deterministic interfaces}
 
In this subsection we {collect several} properties 
of perturbed interfaces {which are important for the subsequent analysis}.
Assume that the perturbation function $\kappa$ is a {fixed}
deterministic function {in $W^{1,\infty}(\Gamma^0)$, in particular} $\kappa$ is independent of
$\om$. Then $\G^\eps$ is defined by
\begin{equation}\label{equ:pert inter}
\G^{\eps}
:=
\{
\vecx
+
\eps \kappa(\vecx)\vecn^0(\vecx)
:
\vecx\in \G^0
\},
\quad
\eps > 0.
\end{equation}
{As already noticed in Section \ref{sec:sto int}, $\G^\eps$ 
is a closed Lipshitz manifold in $\R^3$ provided 
$0 \leq \eps \leq \eps_0$ and $\eps_0$ is sufficiently small.  
In this case $\G^\eps$ introduces a decomposition of $\R^3$ 
into the interior and exterior subdomains $D_-^\eps$ and $D_+^\eps$, respectively.}

Following~\cite{SokZol92},
we define a mapping $T^\eps:\R^3\goto \R^3$ which transforms 
$\G^0$ into $\G^\eps$ and $D_\pm^0$ into $D_\pm^\eps$, respectively, by 
\begin{equation}\label{equ:T eps define}
T^\eps(\vecx)
:=
\vecx
+
\eps \tilde{\kappa}(\vecx) \tilde{\vecn}{}^0(\vecx),
\quad
\vecx\in \R^3,
\end{equation}
where $\tilde{\kappa}$ and $\tilde{\vecn}{}^0$ are any 
{smoothness-preserving} extensions of $\kappa$ and $\vecn^0$ into $\R^3$. 
{We require in particular that $\tilde\kappa\in W^{1,\infty}(\R^3)$. 
Without loss of generality we} assume that the extension $\tilde\kappa$ 
{vanishes outside a sufficiently large ball $B_{R} := \{\vecx\in\R^3 : |\vecx|<R\}$ 
containing $\Gamma^\eps$ for any $0 \leq \eps \leq \eps_0$. 
This implies that the perturbation mapping $T^{\eps}(\vecx)$ 
is an identity in the complement $B_R^c := \R^3 \setminus \overline{B_R}$, i.e.}
\begin{equation}\label{equ:T e cond}
T^{\eps}(\vecx) 
=
\vecx
\qquad
\forall \vecx
\in  {B_{R}^c}.
\end{equation}
{For the ease of notation we abbreviate} 
\begin{equation}\label{equ:V def}
V(\vecx)
:=
\tilde\kappa(\vecx)\tilde\vecn{}^0(\vecx),
\qquad 
\vecx\in \R^3.  
\end{equation}
In~\cite{SokZol92}, $V$ is called \emph{the velocity 
field} of the mapping $T^\eps$. 
The following result is straightforward.

\begin{lemma}\label{lem:V prop}
{Assuming} $\tilde\kappa\in W^{1,\infty}(\R^3)$ and
$\tilde\kappa(\vecx) = 0$ for ${\vecx \in B_R^c}$, 
there hold $V\in 
\big(H^1(\R^3)\big)^3$ 
and
\begin{equation*}
\frac{\partial^m V(\vecx)}{\partial x_l^m} = \veczero
\quad
\forall \vecx\in {B_{R}^c},
\quad
l = 1,2,3,
\quad
m=0,1.
\end{equation*}
 \end{lemma}

{Recall the definition \eqref{equ:W def} of the weighted 
space $W_\eps$ associated to the splitting $\R^3 = \overline{D_-^\eps} \cup \overline{D_+^\eps}$.}
It can be proved that a function $v$ belongs to 
{$W_\eps$} if and only if the composition $v\circ T^\eps$ belongs to 
{$W_0$}, and there hold
\begin{equation}\label{equ:W Weps}
\begin{aligned}
\norm{(v^\eps)_-}{H^1(D_-^\eps)}
&
\simeq
\norm{(v^\eps\circ T^\eps)_-}{H^1(D_-^0)}
\\
\norm{(v^\eps)_+}{H_{w}^1(D_+^\eps)}
&
\simeq
\norm{(v^\eps\circ T^\eps)_+}{H_{w}^1(D_+^0)}
\\
\norm{v^\eps}{{W_\eps}}
&
\simeq
\norm{v^\eps\circ T^\eps}{W_0}.
\end{aligned}
\end{equation}
In {the subsequent analysis}, for any $3$ 
by $3$ matrix  $A(\vecx)$ 
whose entries are functionals of $\vecx\in U\subset\R^3$,
we denote
\[
\norm{A(\cdot)}{L^p(U)}
:=
\max_{i,j=1,2,3}
\{ 
\norm{A_{i,j}(\cdot)}{L^p(U)}
\},
\quad
{1\le p\le \infty},
\]
where $A_{ij}$ are components of $A$.

The following three lemmas state some important properties of 
the mapping $T^\eps$ which will be used later in this section. 
\ac{
Until the end of this section we assume that $T^\eps$ is defined by~\eqref{equ:T eps define} and~\eqref{equ:T e cond} with 
$\tilde\kappa\in C^{1}(\R^3)$, 
and denote its Jacobian matrix and Jacobian determinant by
$J_{T^\eps}$ and $\gamma(\eps,\cdot)$, respectively.
}

\begin{lemma}\label{lem:T eps prop} 
\ac{Consider} $A(\eps,\cdot):= \gamma(\eps,\cdot) J_{T^\eps}^{-1} J_{T^\eps}^{-\top}$,
where $J_{T^\eps}^{\top}$ \ac{is} the transpose of $J_{T^\eps}$. \ac{Then} there hold
\begin{equation}\label{equ:st 1}
\lim_{\eps\goto 0} \norm{A(\eps,\cdot) - I}{L^\infty(\R^3)} = 0
\end{equation}
and
\begin{equation}\label{equ:st 2}
\lim_{\eps\goto 0} \norm{\dfrac{A(\eps,\cdot) - I}{\eps} - A'(0,\cdot)}{L^2(\R^3)} = 0.
\end{equation}
Here, $A'(0,\cdot)$ is the G\^ateaux derivative of $A$ 
(determined by $T^\epsilon$) at $\eps =0$, namely
\[
A'(0,\vecx)
=
\lim_{\eps\goto 0}
\frac{A(\eps,\vecx) - I(\vecx)}{\eps},
\quad\vecx\in \R^3.
\]

\end{lemma}

\begin{proof}
Denoting $V(\vecx):= 
(V_1(\vecx), V_2(\vecx), V_3(\vecx))^\top$, the Jacobian matrix and the Jacobian 
of
$T^\eps$ are given by
\begin{equation}\label{equ:Jc T}
J_{T^\eps}(\vecx)
=
\begin{bmatrix}
1 + \eps\dfrac{\partial V_1(\vecx)}{\partial x_1}
&
\eps\dfrac{\partial V_1(\vecx)}{\partial x_2}
&
\eps\dfrac{\partial V_1(\vecx)}{\partial x_3}
\\
\eps\dfrac{\partial V_2(\vecx)}{\partial x_1}
&
1+\eps\dfrac{\partial V_2(\vecx)}{\partial x_2}
&
\eps\dfrac{\partial V_2(\vecx)}{\partial x_3}
\\
\eps\dfrac{\partial V_3(\vecx)}{\partial x_1}
&
\eps\dfrac{\partial V_3(\vecx)}{\partial x_2}
&
1+\eps\dfrac{\partial V_3(\vecx)}{\partial x_3}
\end{bmatrix}
\end{equation}
and 
\begin{align}
\gamma(\eps,\vecx)
&
=
{\Big{|}}
1 + \eps
\Big(
\sum_{k = 1}^3
\dfrac{\partial V_k(\vecx)}{\partial x_k}
\Big)
+
\eps^2
\Big(
\sum_{k,l = 1\atop k\not= l}^3
\dfrac{\partial V_k(\vecx)}{\partial x_k}
\dfrac{\partial V_l(\vecx)}{\partial x_l}
-
\dfrac{\partial V_l(\vecx)}{\partial x_k}
\dfrac{\partial V_k(\vecx)}{\partial x_l}
\Big)
\notag
\\
&
\quad +
\eps^3
\Big(
\sum_{i,j,k = 1}^3
{\rm sign}{(i,j,k)}
\dfrac{\partial V_i(\vecx)}{\partial x_1}
\dfrac{\partial V_j(\vecx)}{\partial x_2}
\dfrac{\partial V_k(\vecx)}{\partial x_3}
\Big)
{\Big{|}}
\notag
\\
&
=:
{\big{|}}
1 + \eps\gamma_1(\vecx) + \eps^2\gamma_2(\vecx) + \eps^3\gamma_3(\vecx)
{\big{|}}.
\label{equ:tt 1}
\end{align}
{Here ${\rm sign}(i,j,k)$ denotes the sign of the permutation $(i,j,k)$.}
The entries $A_{ij}(\eps,\vecx)$, $i,j = 1,2,3$, 
of the matrix $A(\eps,\vecx)$
are given by
\begin{equation}\label{equ:A eps comp}
A_{ij}(\eps,\vecx)
=
\gamma(\eps,\vecx)^{-1}
\left(
\delta_{ij}
+
\sum_{n=1}^4 \eps^n h_{ijn}(\vecx)
\right),
\end{equation}
where $h_{ijn}$ is a polynomial of 
partial derivatives of $V$ and 
$\delta_{ij}$ is the Kronecker delta. 
Using Lemma~\ref{lem:V prop}, we deduce 
\begin{equation}\label{equ:st 4}
\begin{gathered}
\gamma_n,\ h_{ijn}\in L^\infty(\R^3)\cap L^2(\R^3),
\quad
i,j =1,2,3
\
\text{and}
\
n=1,\ldots,4,
\\
\lim_{\eps\goto 0}
\norm{\gamma(\eps,\cdot)}{L^\infty(\R^3)}
>0,
\end{gathered}
\end{equation}
where $\gamma_1$, $\gamma_2$, $\gamma_3$ are defined by~\eqref{equ:tt 1} and 
$\gamma_4:= 0$ for notational convenience later.
In particular, for sufficiently small $\eps>0$, there holds
\begin{equation}\label{equ:vbn 4}
\gamma(\eps,\vecx)
=
1 + \eps\gamma_1(\vecx) + \eps^2\gamma_2(\vecx) + \eps^3\gamma_3(\vecx) 
\geq c > 0
\qquad
\forall \vecx\in \R^3.
\end{equation}
Consider from now on sufficiently small $\eps>0$.
It follows from~\eqref{equ:A eps comp} and~\eqref{equ:vbn 4} that
the $ij$-entry of the matrix $A(\eps,\vecx) - I$ is 
\begin{align}
&
A_{ij}(\eps,\cdot) 
-
{\delta_{ij}}
=
\eps\, \gamma(\eps,\cdot)^{-1} 
\sum_{n=1}^4 \eps^{n-1}\big(h_{ijn} - \delta_{ij}\gamma_n\big).
\label{equ:st 3}
\end{align}
Hence, \eqref{equ:st 4} yields
\[
\norm{A_{ij}(\eps,\cdot)-{\delta_{ij}}}{L^\infty(\R^3)}
\goto 0
\quad
\text{as}
\quad 
\eps\goto 0,
\]
proving~\eqref{equ:st 1}.

From~\eqref{equ:st 3}, we  have 
\begin{align}
&
\frac{A_{ij}(\eps,\cdot) 
-
{\delta_{ij}}}{\eps}
=
\gamma(\eps,\cdot)^{-1} 
\sum_{n=1}^4 \eps^{n-1}\big(h_{ijn} - \delta_{ij}\gamma_n\big).
\label{equ:st 5}
\end{align}
Taking the limit when $\eps$ goes to $0$,
noting that $\gamma(\eps,\cdot)\goto 1$,
we obtain 
\begin{equation}\label{equ:st 6}
A_{ij}'(0,\cdot)
=
h_{ij1}-\delta_{ij}\gamma_1,
\quad
i,j=1,2,3.
\end{equation}
Subtracting~\eqref{equ:st 6} from~\eqref{equ:st 5} side by side,
 we obtain
\begin{align}
\frac{A_{ij}(\eps,\cdot) 
-
{\delta_{ij}}}{\eps}
-
A_{ij}'(0,\cdot) 
&
=
\gamma(\eps,\cdot)^{-1}
\Big(
\sum_{n=2}^4
\eps^{n-1}
(h_{ijn}-\delta_{ij}\gamma_n)
-
(h_{ij1}-\delta_{ij}\gamma_1)
(\gamma(\epsilon, \cdot)-1)
\Big).
\label{equ:apr 4}
\end{align}
Noting~\eqref{equ:st 4}, we infer
\[
\lim_{\eps\goto 0}
\norm{\frac{A_{ij}(\eps,\cdot) 
-
{\delta_{ij}}}{\eps}
-
A_{ij}'(0,\cdot) 
}{L^2(\R^3)}
=
0,
\]
proving~\eqref{equ:st 2}. 
\end{proof}

\begin{lemma}\label{lem:f T eps}
For any function $v\in L^2(\R^3)$, there holds
\begin{equation*}
\lim_{\eps\goto 0} 
\norm{\sqrt{1 + \snorm{\,\cdot\,}{×}^2}
\Big(
\ga(\eps,\cdot)\, v\circ T^\eps - v\Big)}{L^2(\R^3)} 
= 
0.
\end{equation*}
\end{lemma}

\begin{proof}
Since $T^\eps(\vecx) = \vecx$ for any $\vecx\in  {B_{R}^c}$, 
see~\eqref{equ:T e cond}, 
the 
Jacobian 
satisfies
\begin{equation}\label{equ:gamma eq 1}
\gamma(\eps,\vecx) = 1
\quad
\text{for any}
\quad
\vecx\in {B_{R}^c}. 
\end{equation}
Therefore,
\begin{align}
\Big{\|}
\sqrt{1 + \snorm{\,\cdot\,}{×}^2}\big(\ga(\eps,\cdot)
-1\big)(v\circ T^\eps)
\Big{\|}_{L^2(\R^3)}
&
=
\norm{\sqrt{1 + \snorm{\,\cdot\,}{×}^2}
\big(\ga(\eps,\cdot)-1\big)(v\circ T^\eps)}{L^2({B_R})}
\notag
 \\
 &
\le
\sqrt{1 + {R^2}}
\norm{\gamma(\eps,\cdot) -1}{L^\infty(\R^3)}\,
\norm{v\circ T^\eps}{L^2(\R^3)}
\notag
\\
&
\leq
C\eps\,
\norm{v\circ T^\eps}{L^2(\R^3)}.
\notag
\end{align}
Using the change of variables $\vecy = T^\eps(\vecx)$ and 
noting~\eqref{equ:st 4}, we have 
\begin{align}
\norm{v\circ T^\eps}{L^2(\R^3)}^2
&
=
\int_{\R^3}
\snorm{v(\vecy)}{×}^2
\big(\gamma(\eps, (T^\eps)^{-1}(\vecy))\big)^{-1}
\,d\vecy
\le 
C
\norm{v}{L^2(\R^3)}.
\notag
\end{align}
Therefore,
\begin{equation}\label{equ:st 9}
\lim_{\eps\goto 0}
\Big{\|}
\sqrt{1 + \snorm{\,\cdot\,}{×}^2}\big(\ga(\eps,\vecx)
-1\big)(v\circ T^\eps)
\Big{\|}_{L^2(\R^3)}
=
0.
\end{equation}

Furthermore, \eqref{equ:T e cond} also gives 
\begin{align}
\norm{\sqrt{1 + \snorm{\,\cdot\,}{×}^2}\big(v\circ T^\eps
- v\big)}{L^2(\R^3)}
&
=
\norm{\sqrt{1 + \snorm{\,\cdot\,}{×}^2}\big(v\circ T^\eps 
- v\big)}{L^2({B_R})}
\le\sqrt{1+{R^2}}\,
\norm{v\circ T^\eps - v}{L^2({B_R})}.
\notag
\end{align}
Note that $\lim_{\eps\goto 0} \norm{v\circ T^\eps - v}{L^2({B_R})}
=0$ if $v$ is continuous. By using a density argument we deduce that 
$\lim_{\eps\goto 0} \norm{v\circ T^\eps - v}{L^2({B_R})} = 0$ for 
$v\in L^2({B_R})$.
Hence, 
\[
\lim_{\eps\goto 0}
\norm{\sqrt{1 + \snorm{\,\cdot\,}{×}^2}\big(v\circ T^\eps
- v\big)}{L^2(\R^3)}
=
0. 
\]
The above identity and~\eqref{equ:st 9} together with the 
{triangle} inequality give the required result.
\end{proof}

\begin{lemma}\label{lem:T eps f 2}
For any function $v\in
H^1(\R^3)$, there holds
\begin{equation*}
\lim_{\eps\goto 0}
\norm{\sqrt{1 + \snorm{\,\cdot\,}{×}^2}
\left(
\frac{\ga(\eps,\cdot) (v\circ T^\eps) - v}{\eps}
-
\divv\big{(} vV \big{)}
\right)
}{L^2(\R^3)}
=
0.
\end{equation*}

\end{lemma}

\begin{proof}
Noting~\eqref{equ:T e cond}, Lemma~\ref{lem:V prop}, 
\eqref{equ:gamma eq 1} and the {triangle}
inequality, we obtain
\begin{align}
\Big{\|}
\sqrt{1 + \snorm{\,\cdot\,}{×}^2}
&
\Big(
\frac{\ga(\eps,\cdot) (v\circ T^\eps)
- v}{\eps}
-
\divv (vV) \Big{)}
\Big{\|}_{L^2(\R^3)}
\notag
\\
&
=
\norm{
\sqrt{1 + \snorm{\,\cdot\,}{×}^2}
\Big(
\frac{\ga(\eps,\cdot) (v\circ T^\eps) - v}{\eps}
-
\divv\big{(} vV \big{)}
\Big)
}{L^2({B_R})}
\notag
\\
&
\lesssim
\norm{
\frac{\ga(\eps,\cdot) (v\circ T^\eps) - v}{\eps}
-
\divv\big{(} vV \big{)}
}{L^2({B_R})}
\notag
\\
&
\le
\norm{\frac{\gamma(\eps,\cdot)-1}{\eps}(v\circ T^\eps) 
-
v\divv V}{L^2({B_R})}
+
\norm{\frac{v\circ T^\eps - v}{\eps} 
-
V\cdot \nabla  v
}{L^2({B_R})}.
\label{equ:st 14}
\end{align}
Recall from~\eqref{equ:tt 1} that $\gamma_1=\divv V$. 
It follows from~\eqref{equ:vbn 4} that
\begin{align}
\frac{\gamma(\eps,\cdot) - 1}{\eps}(v\circ T^\eps)
-
v \divv V
&
=
\gamma_1
(v\circ T^\eps -v)
+
\eps(\gamma_2 + \eps\gamma_3)
(v\circ T^\eps).
\notag
\end{align}
Employing the density argument as in proof of Lemma~\ref{lem:f T eps}, 
we obtain
\begin{equation*}
\lim_{\eps\goto 0}
\norm{\gamma_1(v\circ T^\eps -v)
}{L^2({B_R})}
=
0
\quad 
\text{and}
\quad
\lim_{\eps\goto 0}
\norm{\eps(\gamma_2 + \eps\gamma_3)(v\circ T^\eps)}{L^2({B_R})}
=
0,
\end{equation*}
so that 
\[
\lim_{\eps\goto 0}
\norm{\frac{\gamma(\eps,\cdot)-1}{\eps}(v\circ T^\eps) 
-
v\divv V}{L^2({B_R})}
=
0.
\]
The second term on the right hand side of~\eqref{equ:st 14}
also {tends} to zero by a density argument, 
noting that $V = 
\partial T^\eps/\partial\eps$ at $\eps =0$. {This completes} the proof of the lemma.
\end{proof}

\subsection{Material and shape derivatives}\label{subsec:mat shape der}
In this subsection, for notational convenience we use the notation 
$D^\eps$ for $D_-^\eps$ or $D_+^\eps$, and $\mathcal{H}^1(D^\eps)$
for $H^1(D_-^\eps)$ or $H_w^1(D_+^\eps)$.

\begin{definition}\label{def:Mat-ShapeDeriv}
For any sufficiently small $\eps$, let $v^\eps$ be an element in 
$\cH^1(D^\eps)$ or $H^{1/2}(\G^\eps)$. 
The material derivative of $v^\eps$, denoted by $\dot v$, is defined by
\begin{equation}\label{MaterialDeriv-def}
\dot v
:=
\lim_{\eps\goto 0}
\frac{v^\eps\circ T^\eps - v^0}{\eps},
\end{equation}
if the limit exists in the corresponding space 
$\cH^1(D^0)$ or $H^{1/2}(\G^0)$.
The \textit{shape derivative} of $v^\eps$ is defined \nolinebreak[9] by 
\begin{equation}\label{ShapeDeriv-def}
v'
=
\begin{cases}
\dot v - \nabla v^0\cdot V
& \text{if}\ v^\eps\in \cH^1(D^\eps),
\\
\dot v - \nabla_{\G^0} v^0\cdot V 
& \text{if}\ v^\eps\in H^{1/2}(\G^\eps),
\end{cases}
\end{equation}
where $\nabla_{\G^0}$ denotes the surface gradient.
\end{definition}

\begin{lemma}\label{lem:v shap mat K}
If $v'$ is a shape derivative of $v^\eps\in \cH^1(D^\eps)$, then for 
any compact set $K\subset\subset D^0$ we have 
\begin{equation}\label{ShapeDeriv-limit}
v'
=
\lim_{\eps\goto 0}
\frac{v^\eps-v^0}{\eps}\quad
\text{in}
\quad
H^1(K).
\end{equation}
\end{lemma}

\begin{proof}
Given $K\subset\subset D^0$, 
there exists an $\eps_0>0$ such that $K\subset\subset D^\eps$ for all 
$0\le \eps\le \eps_0$.
We denote by $\mathcal T: [0,\eps_0]\times\R^3\goto\R^3$ the mapping given by
\[
\mathcal T(\eps,\vecx)
:=
T^\eps(\vecx),
\quad
\forall 
(\eps,\vecx)
\in 
[0,\eps_0]\times\R^3.
\]
We also denote  by $\tilde v(\eps, \vecx) := v^\eps(\vecx)$ 
for any $0\le \eps \le \eps_0$ and 
$\vecx\in D^\eps$.
By the definition of material derivative,
we have
\[
\dot{v}
=
\frac{\partial}{\partial\eps}
\tilde v(\eps, \mathcal{T}(\eps,\cdot))\Big{|}_{\eps=0},
\quad
\text{in}
\quad
H^1(K).
\]
Applying the chain  rule, we obtain
\begin{align}
\dot v
&
=
\frac{\partial\tilde v}{\partial \eps}
(0,\mathcal{T}(0,\cdot)) 
+
\nabla \tilde v(0,\mathcal{T}(0,\cdot))
\cdot
\frac{\partial \mathcal{T}(0,\cdot)}{\partial \eps}
\notag
\\
&
=
\frac{\partial \tilde v(0,\cdot)}{\partial \eps}
+
\nabla v^0
\cdot
V,
\quad
\text{in}
\quad
H^1(K).
\notag
\end{align}
This implies 
\[
v'
=
\frac{\partial \tilde v(0,\cdot)}{\partial \eps}
=
\lim_{\eps\goto 0}
\frac{v^\eps-v^0}{\eps}
\quad
\text{in}
\quad
H^1(K).
\]
\end{proof}

\begin{remark}\label{rem:v sha vx}
The limit in the above lemma does not hold in $\cH^1(D^0)$ since 
in general, $v^\eps$ does not belong to $\cH^1(D^0)$.
\end{remark}

Similar definitions can be introduced for vector functions $\vecv$. 
The following lemmas state some useful properties of material and 
shape derivatives which will be used frequently in the remainder
of the paper.

\begin{lemma}\label{pro:mat sha pro}
Let $ \dot v$, $\dot w$ be material derivatives, and
$v'$, $w'$ be shape derivatives of
 $v^\eps$, $w^\eps$ in $\cH^1(D^\eps)$, $\eps\ge 0$, respectively. 
Then the following statements are true.
\begin{enumerate}[(i)]
\item\label{ite:t1}  
The material and shape derivatives of the product $v^\eps w^\eps$ 
are $
\dot v w^0
+
v^0\dot w$ and 
$
v' w^0
+
v^0 w'$, respectively.
\item\label{ite:t4}
The material and shape derivatives of the quotient $v^\eps/ w^\eps$ 
are $
(\dot v w^0
-
v^0\dot w)/(w^0)^2$ and 
$
(v' w^0
-
v^0 w')/(w^0)^2$, respectively,
provided that all the fractions are well-defined.
\item\label{ite:t3}
If $v^\eps = v$ for all $\eps\ge 0$, then $ \dot v = \nabla v^0\cdot V
= \nabla v\cdot V$
and $v' = 0$.
\item\label{ite:t2}
If 
\[
\mathcal{J}_1(D^\eps) := \displaystyle\int_{D^\eps} v^\eps\,d\vecx,
\
\mathcal{J}_2(D^\eps) := \displaystyle\int_{\G^\eps} v^\eps\,d\sigma,
\
\text{and}
\
dJ_i(D^\eps)|_{\eps=0}
:=
\lim_{\eps\goto 0}
\frac{J_i(D^\eps)-J_i(D^0)}{\eps},
\
i=1,2,
\]
then 
\[
d\mathcal{J}_1(D^\eps)|_{\eps=0}
=
\int_{D^0} v'\,d\vecx
+
\int_{\G^0} v^0
\inpro{V}{{\vecn^0}}
\,d\sigma
\]
and 
\[
d\mathcal{J}_2(D^\eps)|_{\eps=0}
=
\int_{\G^0} v'\,d\sigma
+
\int_{\G^0} 
\left(
\frac{\partial v^0}{\partial n}
+
\divv_{\G^0}(\vecn^0)\,
v^0
\right)
\inpro{V}{{\vecn^0}}
\,d\sigma.
\]
\end{enumerate}
\end{lemma}

\begin{proof}
Statements~{\eqref{ite:t1}--\eqref{ite:t3}} can be obtained by using 
elementary calculations. Statement~\eqref{ite:t2}
is proved in~\cite[pages~113, 116]{SokZol92}.
\end{proof}

\begin{lemma}\label{lem:n shape}
The material and shape derivatives of the normal field
$\vecn^\eps$ are given by 
\[
\dot{\vecn} 
=
\vecn' = -\nabla_{\G^0}\kappa.
\]
\end{lemma}

\begin{proof}
We start by recalling that the material and the shape derivative of surface fields are identical in the case of normal surface perturbation \eqref{equ:V def}. Particularly, from \eqref{equ:V def} and \eqref{ShapeDeriv-def} we find
\[
\dot{\vecn} - \vecn' = \nabla_{\G^0} \vecn^0\cdot \kappa \vecn^0 = 0.
\]
Recall that the unit normal vector field $\vecn^\eps$ of the perturbed interface
$\G^\eps$ is related to that of the reference interface $\G^0$ by 
\begin{equation*}
\vecn^\eps{\circ}T^\eps(\vecx)
=
\frac{J_{T^\eps}^{-\top}(T^\eps(\vecx))\, \vecn^0(\vecx)}
{\snorm{J_{T^\eps}^{-\top}(T^\eps(\vecx))\, \vecn^0(\vecx)}{×}}.
\end{equation*}
Therefore,
\begin{align}
\dot\vecn
=
&
\lim_{\eps\goto 0}
\frac{\vecn^\eps{\circ}T^\eps(\vecx) - \vecn^0(\vecx)}{\eps}
\notag
\\
&
=
\left(
\lim_{\eps\goto 0}
\frac{J_{T^\eps}^{-\top}(T^\eps(\vecx)) - I}{\eps}
-
\lim_{\eps\goto 0}
\frac{
\snorm{
J_{T^\eps}^{-\top}(T^\eps(\vecx))\vecn^0(\vecx)}{} - 1}{\eps}
\right)
\lim_{\eps\goto 0}
\frac{\vecn^0(\vecx)}
{\snorm{J_{T^\eps}^{-\top}(T^\eps(\vecx))\,\vecn^0(\vecx)}{×}}
\notag
\\
&
=
\left(
\frac{d J_{T^\eps}^{-\top}(T^\eps(\vecx))}{d\eps}\Big{|}_{\eps=0}
-
\frac{d \snorm{J_{T^\eps}^{-\top}(T^\eps(\vecx))\vecn^0(\vecx)}{}}{d\eps}\Big{|}_{\eps=0}
\right)
\vecn^0(\vecx),
\label{equ:n dot long}
\end{align}
noting from~\eqref{equ:Jc T} that
\[
\lim_{\eps\goto 0}
J_{T^\eps}^{-\top}
=
\lim_{\eps\goto 0}
J_{T^\eps}
=
I.
\]
Since $I = J_{T^\eps}^{-1}(T^\eps(\vecx))\, J_{T^\eps}(\vecx)$
for all $\vecx\in \R^3$, we have 
$\veczero = \frac{d}{d\eps}\big(J_{T^\eps}^{-1} J_{T^\eps}\big)|_{\eps=0}$,
which together with the product rule and~\eqref{equ:Jc T} yields
\begin{equation}\label{equ:d J 1}
\frac{d}{d\eps}\big(J_{T^\eps}^{-\top}(T^\eps(\vecx))\big)\Big{|}_{\eps=0}
=
-
(J_{T^0})^{-\top}
\Big(
\frac{d}{d\eps}
(J_{T^\eps}^{\top})
\Big{|}_{\eps=0}
\Big)
(J_{T^0})^{-1}
=
-\frac{d}{d\eps}
(J_{T^\eps})
\Big{|}_{\eps=0}
=
-
J_V^{\top},
\end{equation}
We also have, using the fact that $\snorm{J_{T^0}^{-\top}\,\vecn^0}{}=1$,
\begin{align}
\frac{d}{d\eps}
\snorm{J_{T^\eps}^{-\top}\,\vecn^0}{×}\Big{|}_{\eps=0}
&
=
\snorm{J_{T^0}^{-\top}\,\vecn^0}{×}
\frac{d}{d\eps}
\snorm{J_{T^\eps}^{-\top}\,\vecn^0}{×}\Big{|}_{\eps=0}
=
\frac{1}{2}
\frac{d}{d\eps}
\Big(
\snorm{J_{T^\eps}^{-\top}\,\vecn^0}{×}^2
\Big)\Big{|}_{\eps=0}
\notag
\\
&
=
\frac{1}{2}
\inpro{\frac{d}{d\eps}\big(J_{T^\eps}^{-1} J_{T^\eps}^{-\top}\big)\,
\vecn^0}{\vecn^0}
=
-
\frac{1}{2}
\inpro{(J_V^{\top} + J_V)\,
\vecn^0}{\vecn^0}.
\label{equ:d J 2}
\end{align}
Simple calculation reveals that 
\begin{equation}\label{equ:dkappa}
J_V^{\top} = \nabla\kappa\,(\vecn^0)^{\top}
\quad
\text{and}
\quad
(J_V^{\top} + J_V)\vecn^0 = \nabla \kappa +\inpro{\nabla\kappa}{\vecn^0×}
\vecn^0.
\end{equation}
Inserting~\eqref{equ:d J 1}--\eqref{equ:dkappa} 
into~\eqref{equ:n dot long}, we obtain
\begin{align}
\dot\vecn
&
=
-J_V^{\top}\,\vecn^0
+
\frac{1}{2}
\inpro{(J_V^{\top} + J_V)\,
\vecn^0}{\vecn^0}\vecn^0
=
-\nabla\kappa
+
\inpro{\nabla\kappa}{\vecn^0} \vecn^0
=
-\nabla_{\G^0}\kappa,
\notag
\end{align}
finishing the proof of the lemma.
\end{proof}

\subsection{Shape derivative of solutions of transmission problem}
In this subsection, we shall discuss the existence of material and 
shape derivatives of the solutions of transmission problems on 
perturbed interfaces.
Consider a deterministic problem {with respect to}
the reference interface $\G^0$:
\begin{subequations}\label{equ:tt 3}
\begin{align}
-\al \triangle u^0 
&
=
f
\quad
\text{in } D_{-}^0\cup D_{+}^0,
\label{equ:b1}
\\
\jum{u^0}
&
=
0
\quad
\text{on }
\G^0,
\label{equ:b2}
\\
\jum{\al\frac{\partial u^0}{\partial \bsb n}}
&
=
0
\quad
\text{on }
\G^0,
\label{equ:b3}
\\
u^0({\vecx}{×})
&
=
\mathcal O(\snorm{\vecx}{×}^{-1})
\quad
\text{when }
\snorm{\vecx}{×}\goto\infty.
\label{equ:b4}
\end{align}
\end{subequations}
The perturbed problem corresponding to  the perturbed 
interface $\G^{\eps}$ is given by

\begin{subequations}\label{equ:tt 4}
\begin{align}
-\al^\eps \triangle u^\eps 
&
=
f
\quad
\text{in } D_{-}^\eps\cup D_{+}^{\eps},
\label{equ:b5}
\\
\jum{u^\eps}
&
=
0
\quad
\text{on }
\G^{\eps},
\label{equ:b6}
\\
\jum{\al^\eps\frac{\partial u^\eps}{\partial \bsb n}}
&
=
0
\quad
\text{on }
\G^{\eps},
\label{equ:b7}
\\
u^\eps({\vecx}{×})
&
=
\mathcal O(\snorm{\vecx}{×}^{-1})
\quad
\text{when }
\snorm{\vecx}{×}\goto\infty,
\label{equ:b8}
\end{align}
\end{subequations}
where (cf. \eqref{equ:alpha cond})
\[
\al^\eps(\vecx)
=
\begin{cases}
\al_-, \quad & \vecx\in D_-^{\eps} \\
\al_+, \quad & \vecx\in D_+^{\eps}.
\end{cases}
\]

\begin{lemma}\label{lem:u eps}
\ac{Suppose} $f\in L^2(\R^3)\cap {W_0^*}$ \ac{and $\kappa\in C^1(\G^0)$,} 
then 
\begin{equation}\label{equ:vbn 9}
\lim_{\eps\goto0}
\norm{u^\eps\circ T^\eps-u^0}{{W_0}} = 0.
\end{equation}
Here, {$W_0^*$} denotes the dual space of {$W_0$} with respect to the 
$L^2$-inner product.
\end{lemma}

\begin{proof}
By multiplying both sides of~\eqref{equ:b5} with an arbitrary function
$v\in C_0^\infty(\R^3)$ and integrating over 
{$D_-^\eps \cup D_+^\eps$}, 
we obtain
\begin{align}
\int_{\R^3}
fv\,d\vecx
&
=
-{\al_-}
\int_{D_-^\eps}
\triangle u^\eps(\vecx)\,v(\vecx)\,d\vecx
-{\al_+}
\int_{D_+^\eps}
\triangle u^\eps(\vecx)\,v(\vecx)\,d\vecx.
\label{equ:fv triangle}
\end{align}
Applying Green's identity  and noting~\eqref{equ:b7}, we obtain
\begin{equation}\label{equ:trl1}
\int_{D_+^\eps\cup D_-^\eps}
\al^\eps(\vecx)\,
\nabla  u^\eps(\vecx)\cdot \nabla  v(\vecx)
=
\inpro{f}{v}_{L^2(\R^3)}
\quad
\forall v\in C_0^\infty(\R^3).
\end{equation}
Since the space $C_0^\infty(\R^3)$ is dense in ${W_\eps}$ 
(see~\cite[Remark~2.9.3]{SauSch11}), there holds
\begin{equation}\label{equ:tr0}
\int_{D_+^\eps\cup D_-^\eps}
\al^\eps(\vecx)\,
\nabla  u^\eps(\vecx)\cdot \nabla  v^\eps(\vecx)
=
\inpro{f}{v^\eps}_{L^2(\R^3)}
\quad
\forall v^\eps\in {W_\eps}.
\end{equation}
Choosing $v^\eps = u^\eps$ gives
\[
\snorm{u^\eps}{{W_\eps}}^2
\simeq
\inpro{f}{u^\eps}_{L^2(\R^3)}
\leq
\norm{f}{{W_\eps^*}}
\norm{u^\eps}{{W_\eps}}.
\]
It follows from~Lemma~\ref{lem:nor snor} that
\begin{equation}\label{equ:br 0}
\norm{u^\eps}{{W_\eps}}
\lesssim
\norm{f}{{W_\eps^*}}
\simeq
\norm{f}{{W_0^*}}.
\end{equation}

On the other hand,
using the change of variables $\vecx = T^\eps(\vecy)$ 
in~\eqref{equ:tr0}, we have 
{(noting that $\al^\eps(T^\eps(\vecy))=\al(\vecy)$)}
\begin{equation}\label{equ:tr1}
\int_{D_+^0\cup D_-^0}
\al(\vecy)
\,
(\nabla  w(\vecy))^{\top}
\,A(\eps,\vecy)\,
\nabla  (u^\eps\circ T^\eps)(\vecy)
\,d\vecy
=
\int_{D_+^0\cup D_-^0}
f(T^\eps(\vecy))\, w(\vecy) \ga(\eps,\vecy)
\,d\vecy,
\end{equation}
for any $w\in {W_0}$.
We also obtain from
problem~\eqref{equ:b1}--\eqref{equ:b4}
\begin{equation}\label{equ:tr2}
\int_{D_+^0\cup D_-^0}
\al(\vecy)
\,
(\nabla  w(\vecy))^{\top}\,
\nabla  u^0(\vecy)\,
d\vecy
=
\int_{D_+^0\cup D_-^0}
f(\vecy)\, w(\vecy) \,d\vecy,
\end{equation}
for any $w\in {W_0}$.
Subtracting~\eqref{equ:tr2} from~\eqref{equ:tr1}
we deduce
\begin{align}
\int_{D_+^0\cup D_-^0}
\al(\vecy)
&
\nabla  w(\vecy)^{\top} \,
\nabla  \Big((u^\eps\circ T^\eps)(\vecy)
-
 u^0(\vecy) \Big)
\,d\vecy
\notag
\\
&
=
-\int_{D_+^0\cup D_-^0}
\al(\vecy)
\big(\nabla w(\vecy)\big)^{\top} \,
\Big( A(\eps,\vecy) -I \Big) \,
\nabla  (u^\eps\circ T^\eps)(\vecy)
\,d\vecy
\notag
\\
&
\quad
+
\int_{D_+^0\cup D_-^0}
\Big(
\ga(\eps,\vecy)f(T^\eps(\vecy)) - f(\vecy)
\Big)
w(\vecy)
\,d\vecy
\qquad
\forall w\in {W_0}.
\label{equ:tr3}
\end{align}
Choosing in~\eqref{equ:tr3} 
{$w = u^\eps\circ T^\eps - u^0$}
gives
{
\begin{align}
\int_{D_+^0\cup D_-^0}
&
\al(\vecy)
\left|
\nabla  \Big((u^\eps\circ T^\eps)(\vecy)
-
 u^0(\vecy) \Big)
\right|^2
\,d\vecy
\notag
\\
&
=
-\int_{D_+^0\cup D_-^0}
\al(\vecy)
\left(\nabla  \Big((u^\eps\circ T^\eps)(\vecy)
-
 u^0(\vecy)
\Big)\right)^{\top}
\Big(
A(\eps,\vecy) -I
\Big)
\nabla 
(u^\eps\circ T^\eps)(\vecy)
\,d\vecy
\notag
\\
&
\quad
+
\int_{D_+^0\cup D_-^0}
\sqrt{1 + \snorm{\vecy}{×}^2}
\Big(
\ga(\eps,\vecy)f(T^\eps(\vecy)) - f(\vecy)
\Big)
\frac{(u^\eps\circ T^\eps)(\vecy) - u^0(\vecy)}%
{\sqrt{1 + \snorm{\vecy}{×}^2}}
\,d\vecy
\notag
\\
&
\lesssim
\norm{\big( A(\eps,\cdot) - I \big)}{L^\infty(\R^3)}
\norm{\nabla (u^\eps\circ T^\eps)}{L^2(\R^3)}
\norm{\nabla  \big(u^\eps\circ T^\eps
-
 u^0 \big)}{L^2(\R^3)}
\notag
\\
&
\quad
+
\norm{\sqrt{1 + \snorm{\cdot}{×}^2}\big(\ga(\eps,\cdot)
f\circ T^\eps - f\big)}{L^2(\R^3)}
\,
\bnorm{\frac{u^\eps\circ T^\eps
-
 u^0}{\sqrt{1 + \snorm{\cdot}{×}^2}}}{L^2(\R^3)}
\notag
\end{align}
}
implying 
\begin{align}
\norm{u^\eps\circ T^\eps - u^0}{{W_0}}
&
\lesssim
\norm{ A(\eps,\cdot) - I }{L^\infty(\R^3)}
\norm{\nabla (u^\eps\circ T^\eps)}{L^2(\R^3)}
\notag
\\
&
\quad
+
\norm{\sqrt{1 + \snorm{\cdot}{×}^2}\big(
\ga(\eps,\cdot)f\circ T^\eps - f\big)}{L^2(\R^3)}.
\notag
\end{align}
Hence, applying Lemma~\ref{lem:T eps prop}, 
noting~\eqref{equ:br 0} and Lemma~\ref{lem:f T eps},
we obtain
\[
\lim_{\eps\goto 0}
\norm{u^\eps\circ T^\eps  - u^0}{{W_0}}
=
0,
\]
finishing the proof of this lemma.
\end{proof}

\begin{lemma}\label{lem:z eps}
Assume that  $f\in H^1(\R^3)\cap {W_0^*}$ \ac{and $\kappa\in C^1(\G^0)$}.
Then, $u^\eps$ has a material derivative belonging to
{$W_0$} which is the 
solution to the following equation with unknown $z$:
\begin{align}
\int_{D_+^0\cup D_-^0}
\al(\vecy)
\,
\nabla  z(\vecy)
\cdot
\nabla   w(\vecy)
\,d\vecy
&
=
-\int_{D_+^0\cup D_-^0}
\al(\vecy)
\nabla 
u^0(\vecy)
A'(0,\vecy)
\big(\nabla w(\vecy)\big)^{\top}
\,d\vecy
\notag
\\
&
\quad
+
\int_{D_+^0\cup D_-^0}
\divv\brac{V(\vecy)f}
w(\vecy)
\,d\vecy
\qquad
\forall w\in {W_0}.
\label{equ:ab 1}
\end{align}
\end{lemma}
\begin{proof}
The {uniqueness and} existence of the solution $z\in {W_0}$ 
to the above equation is confirmed by 
\cite[Theorem 2.10.14]{SauSch11}.
Let $z^\eps := (u^\eps\circ T^\eps - u^0)/\eps$.
Our task is to prove that $\displaystyle \lim_{\eps\goto 0}
\norm{z^\eps-z}{{W_0}}= 0$.
Dividing~\eqref{equ:tr3} by $\eps$ we obtain
\begin{align}
\int_{D_+^0\cup D_-^0}
\al(\vecy)
\,
\nabla  z^\eps(\vecy)
\cdot
\nabla   w(\vecy)
\,d\vecy
&
=
-\int_{D_+^0\cup D_-^0}
\al(\vecy)
\nabla 
(u^\eps\circ T^\eps)(\vecy)
\frac{
A(\eps,\vecy) -I
}{\eps}
\big(\nabla w(\vecy)\big)^{\top}
\,d\vecy
\notag
\\
&
\quad
+
\int_{D_+^0\cup D_-^0}
\frac{
\ga(\eps,\vecy)f(T^\eps(\vecy)) - f(\vecy)
}{\eps}
w(\vecy)
\,d\vecy
\qquad
\forall {w\in {W_0}}.
\label{equ:tr4}
\end{align}
Subtracting~\eqref{equ:ab 1} from \eqref{equ:tr4}
yields
\begin{align}
\int_{D_+^0\cup D_-^0}
&
\al(\vecy)
\,
\nabla \brac{z^\eps(\vecy)-z(\vecy)}
\cdot
\nabla   w(\vecy)
\,d\vecy
\notag
\\
&
=
- \int_{D_+^0\cup D_-^0}
\al(\vecy)
\left(\nabla  (u^\eps\circ T^\eps)(\vecy)\,\frac{A(\eps,\vecy)-I}{\eps} 
-\nabla  u^0(\vecy) A'(0,\vecy)\right)
\cdot 
\nabla  w(\vecy)\,d\vecy
\notag
\\
&
\quad
+
\int_{D_+^0\cup D_-^0}
\left( 
\frac{\ga(\eps,\vecy) f(T^\eps(\vecy)) - f(\vecy)}{\eps}
-
\divv\Big{(} V(\vecy)f(\vecy) \Big{)}
\right)
w(\vecy)\,d\vecy
\notag \\
&
=: I_1(w) + I_2(w).
\label{equ:bv 1}
\end{align}
The first integral in the right hand side of~\eqref{equ:bv 1}
can be written as
\begin{align}
I_1(w)
&
=
\int_{D_+^0\cup D_-^0}
\al(\vecy)\,
\nabla  (u^\eps\circ T^\eps)(\vecy)
\left(\frac{A(\eps,\vecy)-I}{\eps} 
- A'(0,\vecy)\right)
\cdot 
\nabla  w(\vecy)\,d\vecy
\notag
\\
&
+
\int_{D_+^0\cup D_-^0}
\al(\vecy)\,
\nabla \big((u^\eps\circ T^\eps)(\vecy) - u^0(\vecy)\big)
A'(0,\vecy)
\cdot
\nabla  w(\vecy)\,d\vecy,
\notag
\end{align}
which converges to $0$ due to~\eqref{equ:vbn 9}
and Lemma~\ref{lem:T eps prop}.
The second integral in the right hand side of~\eqref{equ:bv 1} 
also converges to $0$ due to Lemma~\ref{lem:T eps f 2}.
Therefore, we have 
\begin{equation}\label{equ:bv 3}
{\lim_{\eps\goto0}}
\int_{D_+^0\cup D_-^0}
\al(\vecy)
\,
\nabla \brac{z^\eps(\vecy)-z(\vecy)}
\cdot
\nabla   w(\vecy)
\,d\vecy
=0
\quad
\forall
w\in {W_0}.
\end{equation}
We choose in~\eqref{equ:bv 1} $w = z^\eps- z$. Then the
absolute value of the
 first integral on the 
right hand side of~\eqref{equ:bv 1} can be {estimated as}
\begin{align}
\snorm{I_1(z^\eps-z)}{×}
&
=
\Big{|}
\int_{D_+^0\cup D_-^0}
\al(\vecy)\,
\nabla  u^\eps(\vecy)
\left(\frac{A(\eps,\vecy)-I}{\eps} 
- A'(0,\vecy)\right)
\cdot 
\nabla  \big( z^\eps(\vecy)-z(\vecy)\big)\,d\vecy
\notag
\\
&
\quad
+
\int_{D_+^0\cup D_-^0}
\al(\vecy)\,
\nabla \big(u^\eps(\vecy) - u^0(\vecy)\big)
A'(0,\vecy)
\cdot
\nabla  
\big( z^\eps(\vecy)-z(\vecy)\big)
\,d\vecy\Big{|},
\notag
\\
&
\lesssim
\norm{\nabla  u^\eps}{L^2(\R^3)}
\,
\norm{\frac{A(\eps,\cdot)-I}{\eps} 
- A'(0,\cdot)}{L^\infty(\R^3)}
\norm{\nabla (z^\eps - z)}{L^2(\R^3)}
\notag
\\
&
\quad
+
\norm{\nabla \big( u^\eps - u^0\big)}{L^2(\R^3)}
\,
\norm{A'(0,\cdot)}{L^\infty(\R^3)}
\norm{\nabla  (z^\eps - z)}{L^2(\R^3)}.
\label{equ:bv 4}
\end{align}
The absolute value of the second integral in~\eqref{equ:bv 1}
when $w =  z^\eps-z$ is bounded by
\begin{equation}\label{equ:cv 1}
\snorm{I_2(z^\eps-z)}{×}
\le
\norm{
\sqrt{1+\snorm{\,\cdot\,}{×}^2}
\left(
\frac{\ga(\eps,\vecy) f(T^\eps(\vecy)) - f(\vecy)}{\eps}
-
\divv\big{(} V(\vecy)f(\vecy) \big{)}
\right)
}{L^2(\R^3)}
\,
\norm{z^\eps-z}{{W_0}}.
\end{equation}
Inequalities~\eqref{equ:bv 4} and~\eqref{equ:cv 1} give
\begin{align}
\norm{z^\eps-z}{{W_0}}
&
\le 
\norm{\nabla  u^\eps}{L^2(\R^3)}
\,
\norm{\frac{A(\eps,\cdot)-I}{\eps} 
- A'(0,\cdot)}{L^\infty(\R^3)}
\notag
\\
&
\quad
+
\norm{\nabla \big( u^\eps - u^0\big)}{L^2(\R^3)}
\,
\norm{A'(0,\cdot)}{L^\infty(\R^3)}
\notag
\\
&
\quad
+
\norm{
\sqrt{1+\snorm{\,\cdot\,}{×}^2}
\left(
\frac{\ga(\eps,\vecy) f(T^\eps(\vecy)) - f(\vecy)}{\eps}
-
\divv\big{(} V(\vecy)f(\vecy) \big{)}
\right)
}{L^2(\R^3)}.
\label{equ:cv 2}
\end{align}
Using this together with ~\eqref{equ:vbn 9} and Lemma~\ref{lem:T eps prop},
we can deduce from~\eqref{equ:cv 2}
\begin{equation}\label{equ:bv 5}
\lim_{\eps\goto 0}
\norm{z^\eps - z}{{W_0}}
=
0. 
\end{equation}
\end{proof}

Hence, we have shown that the solution of the transmission 
problem~\eqref{equ:tt 4} has a material derivative, and
thus a shape derivative.
The latter turns out to be the solution of a transmission
problem on the {nominal} interface $\G^0$.

\begin{lemma}\label{lem:sha der}
Under the assumption of Lemma~\ref{lem:z eps},
the shape derivative $u'$ of $u^\eps$ exists and is  
the solution of the transmission problem
\begin{equation}\label{equ:pro deri}
\begin{cases}
\Delta u' 
&=
0
\quad
\text{in } D_-^0\cup D_+^0
\\
\left[u'\right]
&= g_D
\quad
\text{on } \G^0
\\
\left[\al\dfrac{\partial u'}{\partial\bsb n}\right] 
&=
g_N
\quad
\text{on } \G^0
\\
\snorm{u'(\vecx)}{}
&=
\cO\brac{{\snorm{\vecx}{×}^{-1}}}
\quad
\text{as } \snorm{\vecx}{×}\goto
\infty,
\end{cases}
\end{equation}
where
\[
g_D := -
\jum{\dfrac{\partial u^0}{\partial\bsb n}}
\kappa
\quad
\text{and}
\quad
g_N :=
\nabla_{\G^0}\cdot 
\Big(
\kappa
\jum{
\al\nabla_{\G^0} u^0
}
\Big).
\]
\end{lemma}

\begin{proof} 
Existence of $u'$ is confirmed by Lemma~\ref{lem:z eps}.
In this proof only, for notational convenience, we use $\vecn_\pm^\eps$ to 
indicate the normal vector to $\G^\eps$ pointing outwards $D_\pm^\eps$, 
respectively. Note here that 
$\vecn^\eps=\vecn_-^\eps = -\vecn_+^\eps$.
From~\eqref{equ:tr0} we deduce
\begin{equation}\label{equ:soll 1}
\al_- 
\int_{D_-^\eps}
\nabla u_-^\eps\cdot \nabla  v\,d\vecx
+
\al_+ 
\int_{D_+^\eps}
\nabla u_+^\eps\cdot \nabla  v\,d\vecx
=
\inpro{f}{v}_{L_2(\R^3)}
\qquad
\forall v\in C_0^\infty(\R^3).
\end{equation}
Denoting 
\[
J(D_\pm^\eps) 
:=
\al_\pm
\int_{D_\pm^\eps}
\nabla u_\pm^\eps(\vecx)\cdot \nabla  v(\vecx)\,d\vecx
\]
and using Green's formula, we obtain
\begin{align}
J(D_\pm^\eps) 
&
=
-\al_\pm 
\int_{D_\pm^\eps}
u_\pm^\eps(\vecx) \triangle  v(\vecx)\,d\vecx
+
\al_\pm 
\int_{\G^\eps}
u_\pm^\eps(\vecx)
\frac{\partial v}{{\partial\vecn_\pm}}\,d\sigma
=:
\mathcal{J}_1(D_\pm^\eps) 
+
\mathcal{J}_2(D_\pm^\eps).
\notag
\end{align}
By Lemma~\ref{pro:mat sha pro}, $u'\triangle v$ is the 
shape derivative of  $u^\eps\triangle v$. On the other hand,
by Lemmas~\ref{pro:mat sha pro}--\ref{lem:n shape}, 
the shape derivative of ${\dfrac{\partial v}{\partial\vecn}\bigg|_{\Gamma^\eps}}
=\nabla v\cdot\vecn^\eps$ is 
$-\nabla_{\G^0} v\cdot\nabla_{\G^0}\inpro{V}{\vecn^0}$, so that 
the shape derivative of $\displaystyle u^\eps{\dfrac{\partial v}{\partial\vecn}\bigg|_{\Gamma^\eps}}$
is $\displaystyle u'{\dfrac{\partial v}{\partial\vecn}\bigg|_{\Gamma^0}}
- u^0\Big(\nabla_{\G^0} v\cdot\nabla_{\G^0}\inpro{V}{\vecn^0}\Big)$.
Using Lemma~\ref{pro:mat sha pro},
we deduce
\begin{align}
d\mathcal{J}_1(D_\pm^\eps)|_{\eps=0} 
&
=
-\al_\pm 
\int_{D_\pm^0}
u_\pm'(\vecx) \triangle  v\,d\vecx
-
\al_\pm 
\int_{\G^0}
u^0
\triangle  v
\inpro{V}{\vecn_\pm^0}\,d\sigma
\notag
\end{align}
and
\begin{align}
d\mathcal{J}_2(D_\pm^\eps)|_{\eps=0} 
&
=
\al_\pm
\int_{\G^0}
\left(
u_\pm'
{\frac{\partial v}{\partial\vecn_\pm}}
-
u^0
\Big(
\nabla_{\G^0} v\cdot\nabla_{\G^0}\inpro{V}{\vecn^0}
\Big)
\right)
d\sigma
+
\al_\pm
\int_{\G^0}
{\frac{\partial}{\partial\vecn_\pm}}
\Big(
u^0
{\frac{\partial v}{\partial\vecn_\pm}}
\Big)
\inpro{V}{\vecn_\pm^0}
d\sigma
\notag
\\
&
\quad
+
\al_\pm
\int_{\G^0}
\divv_{{\G^0}}(\vecn_\pm^0)
u^0
{\frac{\partial v}{\partial\vecn_\pm}}
\inpro{V}{\vecn_\pm^0}
\,d\sigma,
\notag
\end{align}
{since} $u_-^0 = u_+^0$ on the interface $\G^0$ by~\eqref{equ:b2}.
Therefore, differentiating by $\eps$ both sides of~\eqref{equ:soll 1},
using Green's
formula, the jump condition~\eqref{equ:b3} and noting that
$\triangle v = \triangle_{\G^0} v + 
{\divv_{\G^0}(\vecn^0)\partial v/\partial\vecn + 
\partial^2v/\partial\vecn^2}$, we obtain
\begin{align}
0
=&
\al_-
\int_{D_-^0}
\nabla u'\cdot\nabla v
\,d\vecx
+
\al_+
\int_{D_+^0}
\nabla u'\cdot\nabla v
\,d\vecx\\
&-
\al_-
\int_{\G^0}
u\inpro{V}{\vecn_-^0} \triangle_{\G^0} v\,d\sigma
\notag
-
\al_+
\int_{\G^0}
u\inpro{V}{\vecn_+^0} \triangle_{\G^0} v\,d\sigma \\
&-
\al_-
\int_{\G^0}
u
\nabla_{\G^0} v\cdot
\nabla_{\G^0}\inpro{V}{\vecn_-^0}
\notag
-
\al_+
\int_{\G^0}
u
\nabla_{\G^0} v \cdot
\nabla_{\G^0}\inpro{V}{\vecn_+^0}.
\notag
\end{align}
Applying the tangential Green formula on the 
third and the fourth integrals on the right hand side of 
the above identity and the product rule, the above 
identity can be written as 
\begin{align}
0 
& 
=
\al_-
\int_{D_-^0}
\nabla u'\cdot\nabla v
\,d\vecx
+
\al_+
\int_{D_+^0}
\nabla u'\cdot\nabla v
\,d\vecx
+
\int_{\G^0}
(\al_-\nabla_{\G^0} u_-^0 - \al_+\nabla_{\G^0} u_+^0)
\cdot 
\nabla_{\G^0} v
\inpro{V}{\vecn_-^0}
\,d\sigma.
\label{equ:tt 5}
\end{align}
We choose in~\eqref{equ:tt 5} $v\in  C_0^\infty(D_\pm)$ to obtain
\begin{equation}\label{equ:soll 3}
\al 
\triangle  u'(\vecx)
=
0,
\quad
\vecx\in
D_\pm^0. 
\end{equation}
We now choose $v\in C_0^\infty(\R^3)$ and applying the 
Green's identity to the first two integrals on the right hand side 
of~\eqref{equ:tt 5}, noting~\eqref{equ:soll 3},
to obtain
\begin{align}
0
&
=
\al_-
\int_{\G^0}
v
\frac{\partial u'_-}{\partial\bsb n_-} 
\,d\sigma
+
\al_+
\int_{\G^0}
v
\frac{\partial u'_+}{\partial\bsb n_+} 
+
\int_{\G^0}
(\al_-\nabla_{\G^0} u_-^0 - \al_+\nabla_{\G^0} u_+^0)
\cdot 
\nabla_{\G^0} v
\inpro{V}{\vecn_-^0}
\,d\sigma.
\end{align}
Applying the tangential Green formula on the surface $\G^0$
to the last term on the right hand side of the above identity, we 
deduce
\[
\int_{\G^0}
v
\jum{
\al
\frac{\partial u'}{\partial\vecn}
} 
\,d\sigma
=
\int_{\G^0}
v
\nabla_{\G^0}\cdot 
\Big(
\inpro{V}{{\vecn_-^0}}
\jum{
\al\nabla_{\G^0} u^0
}
\Big)
\,d\sigma,
\]
yielding
\begin{equation}\label{equ:ju ju cond}
\jum{
\al
\frac{\partial u'}{\partial\bsb n}
} 
=
\nabla_{\G^0}\cdot 
\Big(
\inpro{V}{{\vecn_-^0}}
\jum{
\al\nabla_{\G^0} u^0
}
\Big)
\quad
\text{on}\ \G^0.
\end{equation}
{Recalling the transmission} conditions~\eqref{equ:b6}, we have for any 
smooth function $v$
\[
\int_{\G^{\eps}} \jum{u^\eps}v\,d\sigma 
=
0.
\]
Differentiating by $\eps$ both sides, applying Lemma~\ref{pro:mat sha pro}
we have
\begin{align*}
0
&
=
d
\brac{\int_{\G^{\eps}}\jum{u^\eps}v\,d\sigma}
=
d
\brac{\int_{\G_{-}^{\eps}}{u_-^{\eps}}v\,d\sigma
-
\int_{\G_{+}^{\eps}}{u_+^{\eps}}v\,d\sigma}
\\
&
=
\int_{\G^0} (u_-^0v)'
+
 \int_{\G^0}
\left(
\frac{\partial(u_-^0 v)}{\partial\bsb n_-}
+
\divv_{\G^0}({\vecn_-^0}) 
(u_-^0v)
\right)
\inpro{V}{{\bsb n_-^0}}
\,d\sigma
\\
&
\quad
-
 \int_{\G^0} (u_+^0 v)'
-
\int_{\G^0}
\left(
\frac{\partial(u_+^0 v)}{\partial\bsb n_+}
+
\divv_{\G^0}({\vecn_+^0})
(u_+^0 v)
\right)
\inpro{V}{{\bsb n_+^0}}
\,d\sigma
\\
&
=
\int_{\G^0}
\jum{u'} v\,d\sigma
+
\int_{\G^0}
\jum{\frac{\partial u^0}{\partial\bsb n}} 
v\inprod{V}{{\bsb n_-^0}}\,d\sigma
\\
&
\quad
+
\int_{\G^0}
\jum{u^0}
\left(
\frac{\partial v}{\partial \bsb n_-}
+
\divv_{\G^0}({\bsb n_-^0})
\,v
\right)
\inprod{V}{{\bsb n_-^0}}\,d\sigma
\\
&
=
\int_{\G^0}
\jum{u'} v\,d\sigma
+
\int_{\G^0}
\jum{\frac{\partial u^0}{\partial\bsb n_-}} 
v\inprod{V}{{\bsb n_-^0}}\,d\sigma,
\end{align*}
noting that $\jum{u^0}=0$. Hence, there holds
\begin{equation}\label{equ:trans cond der}
\jum{u'}
=
-
\jum{\frac{\partial u^0}{\partial\bsb n}}
\inpro{V}{{\bsb n_-^0}}
=: g_D.
\end{equation}
Hence, from~\eqref{equ:soll 3}, \eqref{equ:ju ju cond} 
and~\eqref{equ:trans cond der}, the shape derivative 
$u'\in H^1(D_-^0)\times H_w^1(D_+^0)$ is the weak solution of the 
transmission problem~\eqref{equ:pro deri}.
\end{proof}

\subsection{Random interfaces}

In Subsection~\ref{subsec:mat shape der}, 
we have  defined material and shape derivatives in which 
the quantity $\kappa(\vecx)$ does not contain uncertainty.
Since the transmission problem~\eqref{equ:original prob} is posed on a 
domain with a random interface (see~\eqref{equ:rand inter}), 
the shape derivative also depends on $\om$, and 
it is necessary to 
approximate 
the mean and the covariance fields of the random solutions.
The result is given in the following lemma, where we recall the notation
$\cH^1(D_\pm^0)$ indicating $H^1(D_-^0)$ or $H_w^1(D_+^0)$.

\begin{lemma}\label{lem:mean cov} Let $u^\eps(\om)$ be the solution of 
the transmission problem~\eqref{equ:lap equ}--\eqref{equ:inft cond 1}
with the random interface $\G^\eps(\om)$ given by~\eqref{equ:rand inter},
and let $u^0$ denote the solution of the transmission problem
with the reference interface $\G^0$.
Assume that the perturbation {function $\kappa$
belongs to $L^k(\Om, \ac{C^1(\G^0)})$ for an integer $k$}
and $f\in H^1(\R^3)\cap {W_0^*}$.
Then, for any compact subset $K\subset\subset D_\pm^0$,
the expectation and the {$k$-th order central moments} of the solution
$u^\eps(\om)$ can be approximated, respectively, by
\begin{equation}\label{equ:Exp}
{\mE [u^\eps]}
=
u^0
+
o(\eps)
\quad
\text{in}
\quad 
{H^1(K)}
\end{equation}
and
\begin{equation}\label{equ:cov app}
\begin{aligned}
{\cM^k[u^\eps - \mE[u^\eps]]
= 
\eps^k\cM^k[u']
+
o(\eps^k)
\quad
\text{in}
\quad
H^1_{\rm mix}(K^k).}
\end{aligned}
\end{equation}
Moreover
\begin{equation}\label{equ:cor app}
\begin{aligned}
{\cM^k[u^\eps - u^0]
= 
\eps^k\cM^k[u']
+
o(\eps^k)
\quad
\text{in}
\quad
H^1_{\rm mix}(K^k).}
\end{aligned}
\end{equation}
\end{lemma}

\begin{proof}
It follows from Lemmas~\ref{lem:v shap mat K} and~\ref{lem:sha der} that 
\begin{equation}\label{equ:u u dot}
u^\eps(\vecx,\om)
=
u^0(\vecx)
+
\eps
u'(\vecx,\om) 
+
{\eps h(\eps,\vecx,\om)}
\quad
\text{in}
\quad 
H^1(K),
\end{equation}
where $h$ satisfies
{$\displaystyle\lim_{\eps\goto 0} \ac{\norm{h(\eps,\cdot,\cdot)}{L^k(\Om,H^1(K))}}
=0$}. This implies
\begin{align}
{
\mE [u^\eps(\vecx,\cdot)]
=
u^0(\vecx)
+
\eps\mE[u'(\vecx,\cdot)]
+
\eps\mE [h(\eps,\vecx,\cdot)]
\quad
\text{in}
\quad 
H^1(K).
\notag}
\end{align}
Here, $u'$ is the solution of~\eqref{equ:pro deri} in which 
the function $\kappa$ defining $g_D$ and $g_N$ depends on $\om$ 
and satisfies $\mE [\kappa] = 0$; see~\eqref{equ:kappa sym}.
Since $u'$ depends linearly on $\kappa$, there also holds
$\mE [u'] = 0$, yielding~\eqref{equ:Exp}.

{
By the definition of the statistical moments \eqref{moments-def} we have
\[
\begin{split}
\cM^k[u^\epsilon - u^0] - \epsilon^k\cM^k[u'] = \epsilon^k \big(\cM^k[u'+h] - \cM^k[u']\big)
\end{split}
\]
\ac{
and by \cite[Corollary V.5.1]{Yos65}
\[
\|\cM^k[u'+h] - \cM^k[u']\|_{H^1_{\rm mix}(K^k)} 
\leq 
\mE\bigg[\| (u'+h) \otimes \dots \otimes (u'+h) - u' \otimes \dots \otimes u'\|_{H^1_{\rm mix}(K^k)} \bigg] 
=: \cE.
\]
Then by the triangle inequality, binomial formula and H\"older's inequality with $p = \dfrac{k}{j}$ and $q = \dfrac{k}{k-j}$ 
\[
\begin{split}
\|\cE\|_{H^1_{\rm mix}(K^k)} 
&=
\mE\bigg[\| 
\sum_{ \scriptsize
\begin{split}v_i &= u' \text{ or } h, \\
(v_1,\dots,v_k) &\neq (u',\dots,u')
\end{split}
} v_1\ \otimes \dots \otimes v_k\|_{H^1_{\rm mix}(K^k)} \bigg]\\
&\leq  
\sum_{ \scriptsize
\begin{split}v_i &= u' \text{ or } h, \\
(v_1,\dots,v_k) &\neq (u',\dots,u')
\end{split}
} 
\mE\bigg[\| v_1\ \otimes \dots \otimes v_k\|_{H^1_{\rm mix}(K^k )} \bigg]
\\
& =  
\sum_{ \scriptsize
\begin{split}v_i &= u' \text{ or } h, \\
(v_1,\dots,v_k) &\neq (u',\dots,u')
\end{split}
} 
\mE\bigg[\| v_1\|_{H^1(K)}  \dots \| v_k\|_{H^1(K)}\bigg]
\\
&= \sum_{j=1}^k \binom{k}{j} \mE\bigg[\|h\|_{H^1(K)}^j \|u'\|_{H^1(K)}^{k-j}\bigg] \\
&\leq \sum_{j=1}^k \binom{k}{j} \mE\bigg[\|h\|_{H^1(K)}^{jp}\bigg]^{\frac{1}{p}}  \mE\bigg[\|u'\|_{H^1(K)}^{(k-j)q}\bigg]^{\frac{1}{q}} \\
&=\sum_{j=1}^k \binom{k}{j} \mE\bigg[\|h\|_{H^1(K)}^{k}\bigg]^{\frac{j}{k}}  \mE\bigg[\|u'\|_{H^1(K)}^{k}\bigg]^{\frac{k-j}{k}} \\
&=\sum_{j=1}^k \binom{k}{j}  \|h\|_{L^k(\Omega,H^1(K))}^j
\|u'\|_{L^k(\Omega,H^1(K))}^{k-j} \\
&= o(1)
\end{split}
\]
and \eqref{equ:cor app} follows.
An analogous estimate holds for}
\[
\begin{split}
\cM^k[u^\epsilon - \mE[u^\eps]] - \epsilon^k\cM^k[u'] = \epsilon^k \big(\cM^k[u'+(h-\mE[h])] - \cM^k[u']\big).
\end{split}
\]

}

\end{proof}

{The above lemma states in particular that $\cM^k[u^\eps-u^0]$, 
$\cM^k[u^\eps-\mE[u^\eps]]$ and $\eps^k\cM^k[u']$ coinside in 
the limit $\eps \to 0$, indicating that $\eps^k\cM^k[u']$ may 
be a good approximation for $\cM^k[u^\eps-u^0]$ and 
$\cM^k[u^\eps-\mE[u^\eps]]$ {if} $\eps$ is small. 
On the other hand, the task of approximation of 
$\eps^k\cM^k[u']$ is significantly simpler than 
approximation of $\cM^k[u^\eps-u^0]$ or $\cM^k[u^\eps-\mE[u^\eps]]$ 
and reduces to solving the homogeneous transmission problem~\eqref{equ:pro deri}.}

\section{Boundary reduction}\label{sec:boundary reduction}

In this section we briefly recall boundary integral equation
methods to solve~\eqref{equ:pro deri}. We rewrite here this problem 
for convenience.

Find $u'\in H^1(D_-^0)\times H^1_w(D_+^0)$ satisfying
\begin{equation}\label{equ:pro trans}
\begin{cases}
\triangle u'
&
=
0
\quad
\text{in } D^0_{\pm}
\\
\jum{u'}
&
=
g_D(\om)
\quad
\text{on } \G^0
\\
\jum{\al\dfrac{\partial u'}{\partial\bsb n}}
&
=
g_N(\om)
\quad
\text{on } \G^0
\\
|u'(\vecx)|
&
=
\cO\brac{\snorm{\vecx}{×}^{-1}}
\quad
\text{as } \snorm{\vecx}{×}\goto
\infty.
\end{cases}
\end{equation}
The  single and double layer potentials are given by 
\begin{equation}\label{equ:sim lay}
{\tilde \cV w(\vecx)}
=
\int_{\G^0}
\frac{1}{\snorm{\vecx-\vecy}{×}}\,
w(\vecy)
\,d\sigma_{\vecy},
\qquad
{\cW v(\vecx)}
=
\int_{\G^0}
\frac{\partial}{\partial \bsb n_{\vecy}}
\frac{1}{\snorm{\vecx-\vecy}{×}}\,
v(\vecy)
\,d\sigma_{\vecy},
\quad
\vecx\in D^0_{\pm}
\end{equation}
{for $w \in H^{-1/2}(\Gamma^0)$ and $v \in H^{1/2}(\Gamma^0)$.}
The limits of these potentials for $\vecx$ approaching $\G^0$
 are given by (see~\cite[page 14]{HsiWen08})
\begin{align}
\cV u(\vecx)
&
:=
\lim_{\vecy\goto\vecx\atop \vecy\in D^0_\pm}
{\tilde \cV u(\vecy)}
\quad
\text{for }
\vecx\in \G^0,
\label{equ:V Ga}
\\
\cK u(\vecx)
&
:=
\lim_{\vecy\goto\vecx\atop \vecy\in D^0_\pm}
{\cW u(\vecy)} \mp \frac{1}{2} u(\vecx)
\quad
\text{for }
\vecx\in\G^0,
\label{equ:K Ga}
\\
\cK' u(\vecx)
&
:=
\lim_{\vecy\goto\vecx\atop \vecy\in D^0_\pm}
{\bsb n_{\vecx} \cdot \nabla_{\vecy} \tilde \cV u(\vecy) }
\pm \frac{1}{2} u(\vecx)
\quad
\text{for } \vecx\in \G^0,
\label{equ:K d Ga}
\\
\cD u(\vecx)
&
:=
-
\lim_{\vecy\goto\vecx\atop \vecy\in D^0_\pm}
{\bsb n_{\vecx}
\cdot 
\nabla_{\vecy} \cW u(\vecy) }
\quad
\text{for } \vecx\in \G^0.
\label{equ:D Ga}
\end{align}
%
The solution of \eqref{equ:pro trans}
is given by
\begin{equation}\label{equ:u V W}
u'(\vecx)
=
\begin{cases}
{\tilde \cV}(\frac{\partial u'_-}{\partial\bsb n})(\vecx) 
- 
{\cW}u'_-(\vecx),
&
\vecx\in D^0_-,
\\
{\cW}u'_+(\vecx)
-
{\tilde \cV}(\frac{\partial u'_+}{\partial\bsb n})(\vecx), 
&
\vecx\in D^0_+;
\end{cases}
\end{equation}
see e.g.~\cite{HsiWen08}.
The Dirichlet-to-Neumann operators
are
\begin{align}
\cS_- u'_-
&
:=
\frac{\partial u'_-}{\partial \bsb n}
=
\cV^{-1}(\frac{1}{2} I + \cK) u'_-, 
\label{equ:DtN 1}
\\
\cS_+ u'_+
&
=
\frac{\partial u'_+}{\partial \bsb n}
=
\cV^{-1}(\cK - \frac{1}{2} I) u'_+ .
\label{equ:DtN 2}
\end{align}
These equalities together with~\eqref{equ:u V W} 
imply
\begin{equation}\label{equ:u V W 2}
u'(\vecx)
=
\begin{cases}
({\tilde\cV}\cS_- - {\cW})(u'_-)(\vecx) 
=: E_-(u'_-)(\vecx),
&
\vecx\in D^0_-
\\
({\cW} - {\tilde\cV}\cS_+)(u'_+)(\vecx)
=: E_+(u'_+)(\vecx), 
&
\vecx\in D^0_+.
\end{cases}
\end{equation}
The randomness of the interface $\G(\om)$ which is given via the randomness
of the vector field $V(\eps,\vecx,\om)$ implies the randomness
in the solution $u$. From~\eqref{equ:u V W 2}, we have
\[
u'(\vecx,\om)
=
\begin{cases}
E_-(u'_-(\om)|_{\G^0})(\vecx),
&
\vecx\in D^0_-,\\
E_+(u'_+(\om)|_{\G^0})(\vecx),
&
\vecx\in D^0_+.
\end{cases}
\]
Tensorizing and integrating both sides of the above equation, 
we deduce
\begin{equation}\label{equ:u tensor ga}
\Covv[u'](\vecx_1,\vecx_2)
=
\begin{cases}
(E_{-,\vecx_1}\otimes E_{-,\vecx_2})\Corr[u'_-|_{\G^0}](\vecx_1,\vecx_2),
&
\vecx_1, \vecx_2\in D^0_-,
\\
(E_{+,\vecx_1}\otimes E_{+,\vecx_2})\Corr[u'_+|_{\G^0}](\vecx_1, \vecx_2),
&
\vecx_1, \vecx_2\in D^0_+,
\end{cases}
\end{equation}
and in general
\begin{equation}
\cM^k[u'](\vecx_1,\dots,\vecx_k)
=
\begin{cases}
(E_{-,\vecx_1}\otimes \dots \otimes E_{-,\vecx_k})\cM^k[u'_-|_{\G^0}]({\vecx_1},\dots,\vecx_k),
&
{\vecx_1}, \dots,\vecx_k\in D^0_-,
\\
(E_{+,\vecx_1}\otimes \dots \otimes E_{+,\vecx_k})\cM^k[u'_+|_{\G^0}]({\vecx_1},\dots,\vecx_k),
&
{\vecx_1}, \dots,\vecx_k\in D^0_+.
\end{cases}
\end{equation}
{Equation~\eqref{equ:u tensor ga}} suggests that the covariance of the 
solution $u'$ in $D^0_\pm$ can be computed from the 
correlation function of the Dirichlet data {$u'_\pm|_{\G^0}$ on the transmission interface}.

The jump conditions in~\eqref{equ:pro trans} gives 
\begin{equation}\label{equ:u u gD}
u'_-(\om)
=
u'_+(\om) +g_D(\om)
\quad\text{on}
\quad \G^0, 
\end{equation}
and  
\begin{equation}\label{equ:S u p}
\underbrace{(\al_- \cS_- - \al_+ \cS_+)}_{=:\jum{\al \cS}}u'_+(\om)
=
g_N(\om) 
-
(\al_- \cS_-) g_D(\om) 
\quad\text{on}
\quad \G^0.
\end{equation}
We note that for a fixed $\om\in\Om$, the right
hand side $g_N(\om) - (\al_- S_-) g_D(\om) \in H^{-1/2}(\G^0)$. 
The solution $u'_+(\om)$ of~\eqref{equ:S u p} belongs to $H^{1/2}(\G^0)$.
The variational form for~\eqref{equ:S u p} is: 
Find $u'_+(\om)\in H^{1/2}(\G^0)$ satisfying 
\begin{equation}\label{equ:var for Buv}
B(u'_+(\om),v) 
=
\langle g_N(\om) - (\al_- \cS_-) g_D(\om) ,v \rangle
\qquad
\forall v\in H^{1/2}(\G^0), 
\end{equation}
with the bilinear form 
$B(\cdot,\cdot)$
and the duality pairing $\langle \cdot,\cdot \rangle$
given by
\begin{equation} \label{D-def}
B(v,w) 
:=
\int_{\G^0}
(\jum{\al \cS} v) w\,d\sigma
\quad
\text{and}
\quad
\langle g,v \rangle :=
\int_{\G^0} gv\,d\sigma
\quad
\forall v,w\in H^{1/2}(\G^0), \quad g \in H^{-1/2}(\G^0).
\end{equation}

We next show the continuity and ellipticity of the operator $\jum{\al \cS}$
which confirms existence of the unique solution of equation~\eqref{equ:S u p}
for a fixed arbitrary $\om$.
\begin{lemma}\label{lem:elipp alS}
The bilinear form 
$B(\cdot,\cdot):H^{1/2}(\G^0)\times H^{1/2}(\G^0)\goto\R$ is 
bounded, i.e.
\begin{equation}\label{equ:bounded s}
\snorm{B(v,w)}{×}
\le 
C_1
\norm{v}{H^{1/2}(\G^0)} 
\norm{w}{H^{1/2}(\G^0)}
\qquad
\forall v, w\in H^{1/2}(\G^0),
\end{equation}
and $H^{1/2}(\spheroid)$-elliptic, i.e.
\begin{equation}\label{equ:coer Duv}
B(v,v) \ge
C_2
\norm{v}{H^{1/2}(\G^0)}^2
\qquad
\forall v\in H^{1/2}(\G^0),
\end{equation}
where the positive constants $C_1$ and $C_2$ are independent of $v$.
\end{lemma}

\begin{proof}
The boundedness of the bilinear form $B$ is derived directly from the 
boundedness of $\cV^{-1}$ and $\cK$.
To prove ellipticity 
we first note that the hypersingular operator $\cD$ is $H^{1/2}(\G^0)$-semi-elliptic 
for all closed interface $\G^0$, i.e.,
\begin{equation}\label{equ:semi ellip D}
\inprod{\cD v}{v}_{L_2(\G^0)}
\ge 
C
\snorm{v}{H^{1/2}(\G^0)}
\quad
\forall 
v\in H^{1/2}(\G^0);
\end{equation}
see e.g.~\cite[Corollary 6.25]{Steinbach08}. 
The Cauchy data $(u_-,\dfrac{\partial u_-}{\partial\bsb n})$
on $\G^0$ satisfy
\begin{equation}\label{equ:Calderon 1}
\begin{pmatrix}
u_-
\\[1.8ex]
\dfrac{\partial u_-}{\partial\bsb n}
\end{pmatrix}
=
\begin{pmatrix}
\dfrac{1}{2}I - \cK
&
\cV
\\
\cD
&
\dfrac{1}{2}I + \cK'
\end{pmatrix}
\begin{pmatrix}
u_-
\\[1.8ex]
\dfrac{\partial u_-}{\partial\bsb n}
\end{pmatrix}.
\end{equation}
Substituting~\eqref{equ:DtN 1} into the second equation
of~\eqref{equ:Calderon 1} gives
\[
\frac{\partial u_-}{\partial\bsb n}
=
\cD\, u_-
+
(\frac{1}{2}I + \cK')
\cV^{-1}
(\frac{1}{2}I + \cK)\,
u_-
\quad
\text{on}
\quad
\G^0.
\]
This equation and~\eqref{equ:DtN 1} yield
\[
\cS_-
=
\cD 
+ 
(\frac{1}{2}I + \cK')
\cV^{-1}
(\frac{1}{2}I + \cK).
\]
Noting that $\cK'$ is the adjoint operator of $\cK$,
we have
\begin{equation}\label{equ:S v v min}
\inpro{\cS_- v}{v}
=
\inpro{\cD v}{v}
+
\inpro{\cV^{-1}(\frac{1}{2}I + \cK)v}{(\frac{1}{2}I + \cK)v}
\quad
\forall v\in H^{1/2}(\G^0).
\end{equation}
Similarly, the exterior Dirichlet-to-Neumann operator 
$\cS_+$ satisfies
\[
\cS_+
=
-
\cD 
- 
(\frac{1}{2}I - \cK')
\cV^{-1}
(\frac{1}{2}I - \cK)
\]
and
\begin{equation}\label{equ:S v v plu}
\inpro{\cS_+ v}{v}
=
-
\inpro{\cD v}{v}
-
\inpro{\cV^{-1}(\frac{1}{2}I - \cK)v}{(\frac{1}{2}I - \cK)v}
\quad
\forall v\in H^{1/2}(\G^0).
\end{equation}
From~\eqref{equ:S v v min}, \eqref{equ:S v v plu}, \eqref{equ:semi ellip D}
 and noting 
the $H^{1/2}$-ellipticity of the inverse operator of $\cV$, we derive
\begin{align}
\inpro{\jum{\al \cS}v}{v}
&
=
(\al_- + \al_+)
\inpro{\cD v}{v}
+
\al_-
\inpro{\cV^{-1}(\frac{1}{2}I+\cK)v}{(\frac{1}{2}I+\cK)v}_{\G^0}
\notag
\\
&
\hspace{4cm}
+
\al_+
\inpro{\cV^{-1}(\frac{1}{2}I-\cK)v}{(\frac{1}{2}I-\cK)v}_{\G^0}
\notag
\\
&
\gtrsim
(\al_- + \al_+)
\snorm{v}{H^{1/2}(\G^0)}^2
+
\al_-
\norm{(\frac{1}{2} I+\cK)v}{H^{1/2}(\G^0)}^2
+
\al_+
\norm{(\frac{1}{2} I-\cK)v}{H^{1/2}(\G^0)}^2
\notag
\\
&
\gtrsim
\snorm{v}{H^{1/2}(\G^0)}^2
+
\norm{(\frac{1}{2} I+\cK)v}{H^{1/2}(\G^0)}^2
+
\norm{(\frac{1}{2} I-\cK)v}{H^{1/2}(\G^0)}^2.
\label{equ:al S inequality}
\end{align}
Applying the triangle inequality to the last two terms on 
the right hand side of the inequality above, we obtain
\begin{align}
\inpro{\jum{\al\cS}v}{v}
&
\gtrsim
\snorm{v}{H^{1/2}(\G^0)}^2
+
\norm{v}{H^{1/2}(\G^0)}^2
\gtrsim 
\norm{v}{H^{1/2}(\G^0)}^2
\quad
\forall v\in H^{1/2}(\G^0),
\notag
\end{align}
completing the proof of the lemma.
\end{proof}
We consider the tensor product 
operator {$\jum{\al\cS}^{(k)}:= \jum{\al\cS}\otimes \dots \otimes \jum{\al\cS}$} 
which is a linear mapping 
\[
{
\jum{\al\cS}^{(k)} 
:
H^{1/2}_{\rm{mix}}(\G^0\times \dots \times \G^0)
\goto
H^{-1/2}_{\rm{mix}}(\G^0\times \dots \times \G^0),
}
\]
see~\cite[Proposition~2.4]{PetersdorffSchwab06} for 
more details. Tensorization of 
equation \eqref{equ:S u p}
yields for {almost all} $\om\in\Om$
\begin{equation}\label{equ:L k otimes}
{\jum{\al\cS}^{(k)} 
\big{(}u'_+(\om)\otimes \dots \otimes u'_+(\om)\big{)}
=
\otimes_{i=1}^k \big{(}g_N(\om) - (\al_- \cS_-)g_D(\om)\big{)}
\quad
\text{in}\quad {H}_{\rm{mix}}^{-1/2}(\G^0\times\dots\times\G^0).}
\end{equation}
Taking the mean of~\eqref{equ:L k otimes}
yields {a deterministic $k$-th moment problem. In particular, for $k=2$ it reads:}
Find ${\Covv[u'_+]}(\vecx,\vecy)\in H^{1/2}_{\rm{mix}}{(\G^0\times \G^0)}$ satisfying
\begin{align}
(\jum{\al\cS}\otimes \jum{\al\cS})\,
{\Covv[u'_+]}(\vecx,\vecy)
&=
(\nabla_{\G,\vecx}\otimes \nabla_{\G,\vecy})\cdot 
\Big(
{\Covv[\kappa]}(\vecx,\vecy)
\jum{\al\nabla_{\G,\vecx} u^0(\vecx)}
\jum{\al\nabla_{\G,\vecy} u^0(\vecy)}
\Big)
\notag
\\
&
+
\big(
(\al_- \cS_-)\otimes(\al_-\cS_-)
\big)
\Big(
{\Covv[\kappa]}(\vecx,\vecy)
\jum{\frac{\partial u^0(\vecx)}{\partial\bsb n_{\vecx}}}
\jum{\frac{\partial u^0(\vecy)}{\partial\bsb n_{\vecy}}}
\Big)
\notag
\\
&
-
\big(
\nabla_{\G,\vecx}\cdot \otimes (\al_-\cS_-)
\big)
\Big(
{\Covv[\kappa]}(\vecx,\vecy)
\jum{\al\nabla_{\G,\vecx} u^0(\vecx)}
\jum{\frac{\partial u^0(\vecy)}{\partial\bsb n_{\vecy}}}
\Big)
\notag
\\
&
-
\big(
(\al_-\cS_-)\otimes \nabla_{\G,\vecy}\cdot
\big)
\Big(
{\Covv[\kappa]}(\vecx,\vecy)
\jum{\al\nabla_{\G,\vecy} u^0(\vecy)}
\jum{\frac{\partial u^0(\vecx)}{\partial\bsb n_{\vecx}}}
\Big).
\label{equ:u plus ten}
\end{align}
Similarly, we have
\begin{align}
(\jum{\al\cS}\otimes \jum{\al\cS})\,
{\Covv[u'_-]}(\vecx,\vecy)
&=
(\nabla_{\G,\vecx}\otimes \nabla_{\G,\vecy})\cdot 
\Big(
{\Covv[\kappa]}(\vecx,\vecy)
\jum{\al\nabla_{\G,\vecx} u^0(\vecx)}
\jum{\al\nabla_{\G,\vecy} u^0(\vecy)}
\Big)
\notag
\\
&
+
\big(
(\al_+\cS_+)\otimes(\al_+\cS_+)
\big)
\Big(
{\Covv[\kappa]}(\vecx,\vecy)
\jum{\frac{\partial u^0(\vecx)}{\partial\bsb n_{\vecx}}}
\jum{\frac{\partial u^0(\vecy)}{\partial\bsb n_{\vecy}}}
\Big)
\notag
\\
&
-
\big(
\nabla_{\G,\vecx}\cdot \otimes (\al_+\cS_+)
\big)
\Big(
{\Covv[\kappa]}(\vecx,\vecy)
\jum{\al\nabla_{\G,\vecx} u^0(\vecx)}
\jum{\frac{\partial u^0(\vecy)}{\partial\bsb n_{\vecy}}}
\Big)
\notag
\\
&
-
\big(
(\al_+\cS_+)\otimes \nabla_{\G,\vecy}\cdot
\big)
\Big(
{\Covv[\kappa]}(\vecx,\vecy)
\jum{\al\nabla_{\G,\vecy} u^0(\vecy)}
\jum{\frac{\partial u^0(\vecx)}{\partial\bsb n_{\vecx}}}
\Big).
\label{equ:u min ten}
\end{align}
{Denote $g^\kappa_+ := \mE[\otimes_{i=1}^k \big{(}g_N(\om) - (\al_- \cS_-)g_D(\om)\big{)}]$.}
Recalling~\eqref{equ:var for Buv}, the variational {formulation for finding $\cM^k[u'_+]$ reads:
Given $g_+^{\kappa}\in H_{\rm{mix}}^{-1/2}(\G^0\times\dots \times\G^0)$,
find $\cM^k[u'_+]\in H_{\rm{mix}}^{1/2}(\G^0\times\dots \times\G^0)$}
satisfying
\begin{equation}\label{equ:Buv tensor}
{\cB}({\cM^k[u'_+]},v)
=
\inprodd{g_+^{\kappa}}{v}
\qquad
\forall
v\in {H}_{\rm{mix}}^{1/2}({\G^0\times\dots \times\G^0}), 
\end{equation}
where $\cB(\cdot,\cdot) = \inprodd{{\jum{\al\cS}^{(k)}}\cdot}{\cdot}$ 
is a bilinear form and $\inprodd{\cdot}{\cdot}$ 
is the $H_{\rm{mix}}^{-1/2}({\G^0\times\dots \times\G^0})$ -- $ H_{\rm{mix}}^{1/2}({\G^0\times\dots \times\G^0})$ 
duality pairing  obtained by tensorisation of 
$B(\cdot,\cdot)$ and 
$\inpro{\cdot}{\cdot}$ 
from \eqref{D-def}.
Proposition 2.4 in \cite{PetersdorffSchwab06} implies

\begin{lemma}\label{lem:bilinear DDuv}
The bilinear form 
$\cB(\cdot,\cdot): H_{\rm{mix}}^{1/2}({\G^0\times\dots \times\G^0})\times
H_{\rm{mix}}^{1/2}({\G^0\times\dots \times\G^0})\goto \R$ is bounded and 
$H_{\rm{mix}}^{1/2}({\G^0\times\dots \times\G^0})$-elliptic, i.e.,
\begin{equation}\label{equ:bounded DDuv}
{\cB(v,w)}{×}
\le 
C_1\norm{v}{H_{\rm{mix}}^{1/2}({\G^0\times\dots \times\G^0})} 
\norm{w}{H_{\rm{mix}}^{1/2}({\G^0\times\dots \times\G^0})},
\end{equation}
and
\begin{equation}\label{equ:coer DDuv}
C_2\norm{v}{H_{\rm{mix}}^{1/2}({\G^0\times\dots \times\G^0})}^2
\le
\cB(v,v) 
\end{equation}
for all $v,w\in H_{\rm{mix}}^{1/2}({\G^0\times\dots \times\G^0})$.
\end{lemma}
By Lemma~\ref{lem:bilinear DDuv} there exists a unique 
solution of~\eqref{equ:Buv tensor}.

\section{Examples}\label{sec:Examp}

In this section, we consider the transmission 
problem~\eqref{equ:lap equ}--\eqref{equ:inft cond 1}
where the random interface $\G(\om)$ is given by
\[
\G(\om)
=
\{
\vecx + \eps\kappa(\vecx,\om)\vecn(\vecx)
:
\vecx\in\mS
\}.
\]
Here, the reference interface $\G^0$ is the unit sphere $\mS$.
The perturbation parameter $\kappa(\vecx,\om) = a(\om)$, where 
$a(\om)$ is uniformly distributed in $[-1,1]$. The mean value
$\mE [\kappa] = 0$ and the covariance 
${\Covv[\kappa](\vecx,\vecy) = 
\Corr[\kappa](\vecx,\vecy)} = 1/3$. The interface 
$\G(\om)$ is a sphere of radius $R(\om) = 1+\eps a(\om)$. 

\subsection{Analytic example}

Firstly, we choose the right hand side $f$ to be
\[
f(\vecx)
=
\begin{cases}
(4 r_{\vecx}^2 - 1)^2
&
\text{if}
\
0\le r_{\vecx}\le 1/2,
\\
0
&
\text{if}
\
1/2\le r_{\vecx},
\end{cases}
\]
where $r_{\vecx} = \snorm{\vecx}{×}$.
Then solution of the transmission problem with respect to the 
random
interface $\G(\om)$ can be analytically computed as follows:
\begin{equation}\label{equ:num exa 1}
u(\vecx,\om)
=
\begin{cases}
\frac{1}{\al_-}
(
\frac{8}{21} r_{\vecx}^6 - \frac{2}{5} r_{\vecx}^4 + \frac{r_{\vecx}^2}{6}
)
-\frac{3}{105\al_-}r_{\vecx}
-
\frac{23}{840\al_-}
+
\frac{\al_+-\al_-}{105\al_-\al_+ R(\om)}
&
\text{if}
\
0\le r_{\vecx}\le \frac{1}{2},
\\
-
\frac{1}{105\al_- r_{\vecx}}
+
\frac{\al_+-\al_-}{105\al_-\al_+ R(\om)}
&
\text{if}
\
\frac{1}{2}\le r_{\vecx}\le R(\om),
\\
-
\frac{1}{105\al_+ r_{\vecx}}
&
\text{if}
\
R(\om)\le r_{\vecx}.
\end{cases}
\end{equation}
In particular, the exact solution $u^0$ of the transmission problem on the reference
interface $\G^0$ is given by~\eqref{equ:num exa 1} where $R(\om) = 1$, i.e.,
\begin{equation}\label{equ:u0 num exp}
u^0(\vecx)
=
\begin{cases}
\frac{1}{\al_-}
(
\frac{8}{21} r_{\vecx}^6 - \frac{2}{5} r_{\vecx}^4 + \frac{r_{\vecx}^2}{6}
)
-\frac{3}{105\al_-}r_{\vecx}
-
\frac{23}{840\al_-}
+
\frac{\al_+-\al_-}{105\al_-\al_+}
&
\text{if}
\
0\le r_{\vecx}\le \frac{1}{2},
\\
-
\frac{1}{105\al_- r_{\vecx}}
+
\frac{\al_+-\al_-}{105\al_-\al_+ }
&
\text{if}
\
\frac{1}{2}\le r_{\vecx}\le {1},
\\
-
\frac{1}{105\al_+ r_{\vecx}}
&
\text{if}
\
{1}\le r_{\vecx}.
\end{cases}
\end{equation}
Noting~\eqref{equ:num exa 1}
and using simple calculation, we obtain
\begin{equation}\label{equ:Eu1 num}
{\mE [u(\vecx,\cdot)]}
=
\begin{cases}
u^0(\vecx) + 
\frac{\al_+-\al_-}{105\al_-\al_+}
\frac{\ln(1+\eps) - \ln(1-\eps)}{2\eps}
&
\text{if}
\
0\le r_{\vecx} < 1,
\\
u^0(\vecx)
&
\text{if}
\
1 < r_{\vecx}.
\end{cases}
\end{equation}
Elementary calculus reveals that 
$\frac{\ln(1+\eps) - \ln(1-\eps)}{2\eps}
= \sum_{n=1}^\infty
 \frac{\eps^{2n}}{2n+1}$.
Therefore, the mean value $\mE [u]$ in~\eqref{equ:Eu1 num}
agrees with our result~\eqref{equ:Exp} in Lemma~\ref{lem:mean cov}.
The linearized error appears in this example to be $\mathcal{O}(\eps^2)$.

We then compute the covariance of the solution $u$ by  
elementary calculations, noting~\eqref{equ:num exa 1},
to obtain
\begin{equation}\label{equ:cov u num 1}
\Covv_u(\vecx,\vecy)
=
\begin{cases}
\frac{1}{3}\frac{\jum{\al}^2}{(105\al_-\al_+)^2}\,\eps^2 
+
\cO(\eps^4)
&
\text{if }
r_{\vecx}<1 {\mbox{ and }} r_{\vecy} < 1,
\\
0
&
\text{if }
r_{\vecx}>1 {\mbox{ or }}  r_{\vecy} > 1.
\end{cases}
\end{equation}

We test accuracy of our shape calculus method by computing the covariance
of $u$ via covariance of the shape derivative. Noting~\eqref{equ:u0 num exp},
we first solve equations~\eqref{equ:u plus ten} and~\eqref{equ:u min ten}
to obtain {$\Covv[u_+']$ and $\Corr_{u_-'}$}. In this 
example, these equations can be solved exactly and 
\[
{\Covv[u_-']}
=
\frac{1}{3}\frac{\jum{\al}^2}{(105\al_-\al_+)^2}
\quad
\text{and}
\quad
{\Covv[u_+']}
=
0.
\]
Applying~\eqref{equ:u tensor ga}, we obtain
\[
{\Covv[u'](\vecx,\vecy)}
=
\begin{cases}
\frac{1}{3}\frac{\jum{\al}^2}{(105\al_-\al_+)^2}
&
\text{if }
r_{\vecx}<1 {\mbox{ and }} r_{\vecy} < 1
\\
0
&
\text{if }
r_{\vecx}>1 {\mbox{ or }} r_{\vecy} > 1.
\end{cases}
\]
This and~\eqref{equ:cov u num 1} agree with our theoretical 
result~\eqref{equ:cov app} and the linearized error in this 
example is 
$\cO(\eps^4)$.

%
%
%

\subsection{{Numerical example}}

Secondly, we solve the problem~\eqref{equ:lap equ}--\eqref{equ:inft cond 1}
where 
the right hand side $f$ is given by
\begin{align}
f(\vecx)
&
=
2
\,
[x_1^2 + x_2^2 +(x_3-1)^2]^{-1/2}
\,
(1-\snorm{\vecx}{×}^2)
\notag
\\
&
\quad 
-4
\,
[x_1^2 + x_2^2 +(x_3-1)^2]^{-1/2}
\,
(\snorm{\vecx}{×}^2-x_3)
-
6
[x_1^2 + x_2^2 +(x_3-1)^2]^{1/2}.
\label{equ:f num}
\end{align}
The deterministic solution of the transmission problem 
with the reference interface $\G^0 = \mS$ is then 
\begin{equation}\label{equ:u num}
\begin{aligned}
u_-(\vecx)
&
=
\frac{1}{\al_-}\,
[x_1^2 + x_2^2 +(x_3-1)^2]^{1/2}
\,
(1-\snorm{\vecx}{×}^2),
\quad
\vecx\in D^0_-
\\
u_+(\vecx)
&
=
\frac{1}{\al_+}\,
[x_1^2 + x_2^2 +(x_3-1)^2]^{1/2}
\,
(1-\snorm{\vecx}{×}^2),
\quad
\vecx\in D^0_+.
\end{aligned}
\end{equation}
Following the method discussed in Section~\ref{sec:Shape cal}, covariance
of the solution is approximated by covariance of the shape derivative
(see Lemma~\ref{lem:mean cov}), which can be obtained by solving the 
equations~\eqref{equ:u plus ten} and~\eqref{equ:u min ten}.
Note here that these equations are given on  the reference interface 
$\G^0 = \mS$. The right hand sides and the solutions of 
these equations belong to the 
tensor space $H_{\rm mix}^{2-\sigma}({\G^0\times\G^0})$ for any $\sigma >0$.
To solve these equations numerically we use the hyperbolic cross 
tensor approximation spaces of spherical harmonics
{which are defined by
 
\begin{equation}\label{equ:S delta L}
S_p^\delta
:=
\spann
\big\{
\bsb\Ylm: \bsb\ell\in\delta_p,\ m_i= -\ell_i,\dots, \ell_i\ \text{for}
\ i = 1,2
\big\},
\end{equation}
where 
\begin{equation}\label{equ:delta L}
\delta_p
:=
\bigg\{
\bsb\ell = (\ell_1,\ell_1)\in\N^2:
\prod_{i=1}^2{(1+\ell_i)}
\le
1+p
\bigg\}.
\end{equation}
The Galerkin method was used to find the approximate solutions 
$u'_p\in S_p^\delta$ of~\eqref{equ:u plus ten} and~\eqref{equ:u min ten}.
It has been shown in~\cite{ChPham12_prep} that the use of the space 
$S_p^\delta$ 
yields the convergence rate of $p^{-(2-\sigma - t)}$ and demands only 
$\mathcal O\big(p^2 \log p\big)$  unknowns, where $t$ is the order of the Sobolev norm in which the errors are computed. \ac{The same convergence rate $p^{-(2-\sigma - t)}$ is } achieved when using the standard full 
tensor product approximation \ac{of degree $p$}
which meanwhile requires $\mathcal{O}\big(p^{4}\big)$ unknowns.
We then compute 
the variance of $u'(\vecx)$ at {three points 
$\vecx = (0,0,0.2)$, $(0,0,0.5)$ and $(0,0,5)$ inside and outside the unit sphere}. 
The  {convergence curves for the absolute error
\[
|\Varr[u'](\vecx) - \Varr[u'_p](\vecx) |
\]
with respect to the order of the hyperbolic cross $p$ are presented in Fig \ref{fig:1}.
}

\begin{figure}[t]
	\centering
	\includegraphics[width=.6\textwidth]{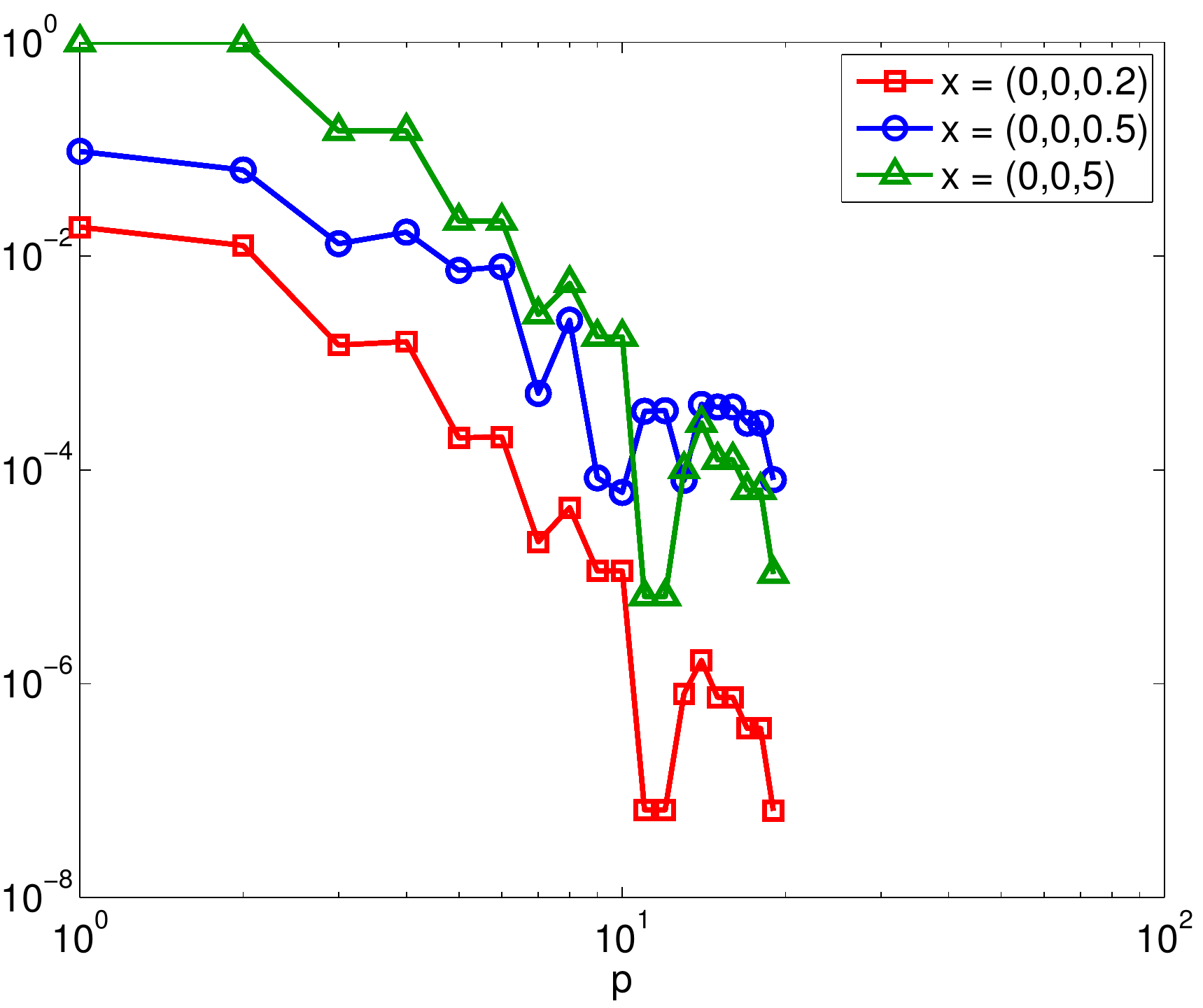}
	\caption{
{Convergence of the absolute error $|\Varr[u'](\vecx) - \Varr[u'_p](\vecx) |$ for three points $\vecx$ inside and outside the unit sphere with respect to the order of the hyperbolic cross $p$.}}
	\label{fig:1}
 \end{figure}

\bibliographystyle{myabbrv}

\begin{thebibliography}{10}

\bibitem{Aubin00}
J.-P. Aubin.
\newblock {\em Applied functional analysis}.
\newblock Pure and Applied Mathematics (New York). Wiley-Interscience, New
  York, second edition, 2000.
\newblock With exercises by Bernard Cornet and Jean-Michel Lasry, Translated
  from the French by Carole Labrousse.

\bibitem{BNobTmp07}
I.~Babu\v{s}ka, F.~Nobile, and R.~Tempone.
\newblock A stochastic collocation method for elliptic partial differential
  equations with random input data.
\newblock {\em SIAM J. Numer. Anal.},  {\bf 45} (2007), 1005--1034.

\bibitem{BarSwbZol11}
A.~Barth, C.~Schwab, and N.~Zollinger.
\newblock Multi-level {M}onte {C}arlo finite element method for elliptic {PDE}s
  with stochastic coefficients.
\newblock {\em Numer. Math.},  {\bf 119} (2011), 123--161.

\bibitem{Ch08sens}
A.~Chernov.
\newblock Abstract sensitivity analysis for nonlinear equations and
  applications.
\newblock In {\em Numerical Mathematics and Advanced Applications}, K.~Kunisch,
  G.~Of, and O.~Steinbach, editors, Proceedings of ENUMATH 2007, Graz, Austria,
  pages 407--414. Springer, Sept. 2008.

\bibitem{ChPham12_prep}
A.~Chernov and T.~D. Pham.
\newblock Sparse spectral {BEM} for elliptic problems with random input data on
  a spheroid.
\newblock {T}echnical {R}eport 1204, INS, University of Bonn, March 2012.

\bibitem{ChSwb13fokm}
A.~Chernov and C.~Schwab.
\newblock First order {$k$}-th moment finite element analysis of nonlinear
  operator equations with stochastic data.
\newblock {\em Math. Comp.},  {\bf 82} (2013), 1859--1888.

\bibitem{CohDVeSwb10}
A.~Cohen, R.~DeVore, and C.~Schwab.
\newblock Convergence rates of best {$N$}-term {G}alerkin approximations for a
  class of elliptic s{PDE}s.
\newblock {\em Found. Comput. Math.},  {\bf 10} (2010), 615--646.

\bibitem{CohDVeSwb11}
A.~Cohen, R.~Devore, and C.~Schwab.
\newblock Analytic regularity and polynomial approximation of parametric and
  stochastic elliptic {PDE}'s.
\newblock {\em Anal. Appl. (Singap.)},  {\bf 9} (2011), 11--47.

\bibitem{ForKor10}
R.~Forster and R.~Kornhuber.
\newblock A polynomial chaos approach to stochastic variational inequalities.
\newblock {\em J. Numer. Math.},  {\bf 18} (2010), 235--255.

\bibitem{Git13adap}
C.~J. Gittelson.
\newblock An adaptive stochastic Galerkin method for random elliptic operators.
\newblock {\em Math. Comp.}, {\bf 82} (2013), 1515--1541.

\bibitem{GraKuoNuySclSlo11}
I.~G. Graham, F.~Y. Kuo, D.~Nuyens, R.~Scheichl, and I.~H. Sloan.
\newblock Quasi-{M}onte {C}arlo methods for elliptic {PDE}s with random
  coefficients and applications.
\newblock {\em J. Comput. Phys.},  {\bf 230} (2011), 3668--3694.

\bibitem{Hrb10output}
H.~Harbrecht.
\newblock On output functionals of boundary value problems on stochastic
  domains.
\newblock {\em Math. Methods Appl. Sci.},  {\bf 33} (2010), 91--102.

\bibitem{HrbLi13}
H.~Harbrecht and J.~Li.
\newblock First order second moment analysis for stochastic interface problems
  based on low-rank approximation.
\newblock {\em ESAIM Math. Model. Numer. Anal.},  {\bf 47} (2013), 1533--1552.

\bibitem{HrbSndSwb08sm}
H.~Harbrecht, R.~Schneider, and C.~Schwab.
\newblock Sparse second moment analysis for elliptic problems in stochastic
  domains.
\newblock {\em Numer. Math.},  {\bf 109} (2008), 385--414.

\bibitem{HsiWen08}
G.~C. Hsiao and W.~L. Wendland.
\newblock {\em Boundary integral equations}, volume 164 of {\em Applied
  Mathematical Sciences}.
\newblock Springer-Verlag, Berlin, 2008.

\bibitem{LigChe85}
W.~A. Light and E.~W. Cheney.
\newblock {\em Approximation theory in tensor product spaces},
Lecture Notes in Mathematics, 1169. Springer-Verlag, Berlin, 1985. vii+157 pp. 

\bibitem{SauSch11}
S.~A. Sauter and C.~Schwab.
\newblock {\em Boundary element methods}, volume~39 of {\em Springer Series in
  Computational Mathematics}.
\newblock Springer-Verlag, Berlin, 2011.
\newblock Translated and expanded from the 2004 German original.

\bibitem{SwbGit11}
C.~Schwab and C.~J. Gittelson.
\newblock Sparse tensor discretizations of high-dimensional parametric and
  stochastic {PDE}s.
\newblock {\em Acta Numer.},  {\bf 20} (2011), 291--467.

\bibitem{SwbTod06}
C.~Schwab and R.~A. Todor.
\newblock Karhunen-{L}o{\`e}ve approximation of random fields by generalized
  fast multipole methods.
\newblock {\em J. Comput. Phys.},  {\bf 217} (2006), 100--122.

\bibitem{SokZol92}
J.~Soko{\l}owski and J.-P. Zol{\'e}sio.
\newblock {\em Introduction to shape optimization}, volume~16 of {\em Springer
  Series in Computational Mathematics}.
\newblock Springer-Verlag, Berlin, 1992.
\newblock Shape sensitivity analysis.

\bibitem{Steinbach08}
O.~Steinbach.
\newblock {\em Numerical approximation methods for elliptic boundary value
  problems}.
\newblock Springer, New York, 2008.
\newblock Finite and boundary elements, Translated from the 2003 German
  original.

\bibitem{PetersdorffSchwab06}
T.~von Petersdorff and C.~Schwab.
\newblock Sparse finite element methods for operator equations with stochastic
  data.
\newblock {\em Appl. Math.},  {\bf 51} (2006), 145--180.

\ac{
\bibitem{Yos65}
K.~Yosida. 
\newblock Functional Analysis, Die Grundlehren der Mathematischen Wissenschaften, 
\newblock Band 123, Academic Press Inc., New York, 1965.
}

\end{thebibliography}

\end{document}